\documentclass[preprint,12pt]{elsarticle}

\usepackage{amsmath,enumerate}

\usepackage{amssymb,amsopn}
\usepackage{graphicx}
\usepackage{caption}
\usepackage{subcaption}
\usepackage{mathtools}
\usepackage{array}

\usepackage{amsmath}
\usepackage{amsthm}
\theoremstyle{definition}

\usepackage{appendix}
\usepackage{pdflscape}
\usepackage{afterpage}
\usepackage{longtable}

\usepackage{geometry}
\usepackage{calc}

\newtheorem{theorem}{Theorem}
\newtheorem{lemma}[theorem]{Lemma}
\newtheorem{proposition}[theorem]{Proposition}
\newtheorem{corollary}[theorem]{Corollary}

\setcitestyle{square,comma}

\begin{document}
\title{\bf Diffuse scattering on graphs}
\author[umich_math]{Anna C. Gilbert}
\ead{annacg@umich.edu}
\author[umich_math]{Jeremy G. Hoskins \corref{cor1}}
\ead{jhoskin@umich.edu}
\author[umich_math,umich_phys]{John C. Schotland}
\ead{jcsch@umich.edu}

\address[umich_math]{Department of Mathematics, University of Michigan, Ann Arbor, MI 48109\\}
\address[umich_phys]{Department of Mathematics and Department of Physics, University of Michigan, Ann Arbor, MI 48109\\} 

\cortext[cor1]{Corresponding author}

\begin{abstract}
We formulate and analyze difference equations on graphs analogous to time-independent diffusion equations arising in the study of diffuse scattering in continuous media.  Moreover, we show how to construct solutions in the presence of weak scatterers from the solution to the homogeneous (background problem) using Born series, providing necessary conditions for convergence and demonstrating the process through numerous examples. In addition, we outline a method for finding Green's functions for Cayley graphs for both abelian and non-abelian groups. Finally, we conclude with a discussion of the effects of sparsity on our method and results, outlining the simplifications that can be made provided that the scatterers are weak and well-separated.

\end{abstract}

\maketitle

\section{Introduction}
Spectral graph theory is a rich and well-developed theory for both the combinatorial and analytic properties of graphs. The following set-up is generally considered. Let $G=(V,E)$ be a graph with vertex set $V$ and edge set $E,$ and $L$ be the combinatorial Laplacian $L,$ or some suitably rescaled variant \cite{specgraphfan}. We can then formulate a graph analog of Poisson's equation
 \begin{equation}\label{eq_int_pois}
\left\{
     \begin{array}{lr}
(Lu)(x)  = f(x), \quad &x \in V\\
u(x) = g(x), \quad &x \in \delta V
\end{array}
\right.
\end{equation}
where $\delta V$ is the set of boundary vertices, which will be discussed in more detail later, and the functions $f$ and $g$ represent internal and boundary sources, respectively. Equation (\ref{eq_int_pois}) has been studied both when the edges are equally-weighted and when the edge weighting varies throughout the graph \cite{ convanal,specgraphfan, discgreen}. In this work, we consider the effect of introducing inhomogeneities on the vertices rather than on the edges, as represented by the addition of a (vertex) potential term to equation (\ref{eq_int_pois}). We call this problem {\it diffuse scattering on graphs} because of its analogy to a related problem in the continuous setting, where the vertex potential is often called the {\it absorption}. A similar problem arises in the study of Schr\"{o}dinger operators on graphs, see for example \cite{vert,eig_path,disc_op,polyom,pert,soardi}. In order to develop the necessary foundations to formulate corresponding inverse problems, which will be analyzed in subsequent works, we also study the role of boundary conditions on the solutions. In particular, we consider Dirichlet, Neumann and Robin, or mixed, boundary conditions, which are often employed in the continuous setting. 

The graph analog of Poisson's equation is related to the classical problem of resistor networks first studied by Kirchhoff in 1847 \cite{kirch}. In that setting, one is given a collection of interconnected resistors to which a voltage source is attached at various points \cite{morrow}. The resulting system can be thought of as a weighted graph, with each edge corresponding to a particular resistor and the vertices representing the connections between them \cite{morrow}. In the event that all the resistors are identical, the voltage at each point satisfies Poisson's equation on the associated graph \cite{dirac_poisson}. In this setting, one seeks either to map the network, finding its corresponding graph \cite{morrow}, solely by measuring the current or potential at various points in the network. This physical analogy is also employed for graph sparsification \cite{Spielman:2008}, as well as in near linear-time solvers for symmetric, diagonally dominant linear systems \cite{DrineasMahoney:2010,KoutisMillerPeng:2011,vishnoi2012laplacian}.

Discrete analogs of PDEs on graphs are not limited to Poisson-type problems and are used extensively in lattice dynamics where we consider the graph analog of the Helmholtz equation \cite{martin, PhysRevLatt}, which arises when considering the Fourier transform of the wave equation. In lattice theory, one problem of particular importance is to examine the propagation of phonons through a crystal in order to determine the size and location of imperfections \cite{martin, PhysRevLatt}.  

In this paper we consider a graph analog of a different PDE, which we call the problem of diffuse scattering on graphs, though the equation also arises in the study of discrete Schr\"{o}dinger operators \cite{disc_op} . A key component of our analysis will be constructing methods for obtaining the appropriate Green's functions for the problems we wish to consider. The idea of discrete Green's functions has, implicitly or explicitly, a long history arising in many important problems and fields such as the study of inverses of tri-diagonal matrices \cite{gantmach}, potential theory \cite{beurl,deny,duffin}, the study of Schr\"{o}dinger operators on graphs \cite{part_dat,eig_path,disc_op,polyom,pert,cartier,soardi,inf_yamasaki}, and the graph-theoretic analog of Poisson's equation \cite{discgreen,inf_kayano, inf_urak2,inf_urak1}. Additionally, Green's function methods have yielded interesting results in many areas including the properties of random walks \cite{discgreen, span_tree}, chip-firing games \cite{RB_Ellis_Thesis}, analysis of online communities \cite{blog}, machine learning algorithms \cite{com_tim,mach_lrn} and load balancing in networks \cite{Chau_Load_Balance}.

As in the analogous continuous problem, we are particularly interested in systems with nearly uniform absorption. By this we mean that the variations in the absorption are small relative to the mean and are typically limited to a small subset of vertices. By defining and applying a discrete version of the Born series we obtain, under suitable conditions, a series solution to the forward problem for a heterogeneous medium, given in terms of the Green's function for the diffusion equation on the same graph but with uniform absorption, called the {\it background Green's function}. We then provide sufficient conditions on the inhomogeneities for the series solution to converge to the correct solution and provide estimates for the rate of convergence.

Although there are many similarities between the equations considered here and those previously mentioned, changing the underlying differential operator gives rise to significant differences in the qualitative behaviour of the solutions. In Section \ref{sec_prelim} we illustrate these differences through various examples and connect the results with their continuous counterparts when such analogues exist. We also consider the important special case of graphs with boundaries, since in applications measurements are carried out on the boundary. In cataloguing the possible boundary conditions, we discuss the well-known Dirichlet and Neumann boundary conditions before formulating a graph equivalent of Robin boundary conditions, similar to those considered in \cite{journ_func}. The introduction of the added parameter representing the mixture of Dirichlet and Neumann boundary conditions will be useful in subsequent work when we consider the inverse problem.

In Section \ref{sec_born} we develop the necessary tools to construct the Born series from the background Green's function. In particular, we prove necessary conditions for the convergence of the series, and discuss the dependence of the rate of convergence on the structure of the graph.

Before applying the Born series to a specific graph, it is first necessary to obtain the background Green's function. In Section \ref{sec_examp} we provide examples of various families of graphs for which the background Green's function is explicitly known and in Section \ref{sec_rep_theory} we discuss the connection between the symmetries of vertex-transitive graphs and group representation theory, showing how to use knowledge of the symmetry group of a graph to obtain an expression for the corresponding background Green's function.

In Section \ref{sec_num_exp} we present a few representative numerical experiments demonstrating the convergence of the Born series for small perturbations to the absorption and compare the empirical convergence results to the bounds obtained in Section \ref{sec_born}. Finally, in Section \ref{sec_pt_abs}, we consider the discrete analogue of a classical problem in scattering theory; the scattering due to a small collection of point absorbers. In the case where there are only one or two point absorbers present, we explicitly sum the Born series and give exact formulae for the scattered fields provided the Green's function for the homogeneous medium is known. We conclude with a comparison of the scattering of light from point absorbers on infinite one-dimensional and two-dimensional lattice graphs to the well-known formulae for the continuous problem of the same dimensions.
 

\section{Preliminaries}\label{sec_prelim}

\subsection{Time-independent diffusion equations on graphs}
Let $\Gamma= (V',E)$ be a connected locally finite loop-free graph with edge set $E$ and vertex set $V'.$ Given a subset, $V,$ of the vertex set $V',$ we define the {\it vertex boundary} of $V,$ $\delta V,$ by \cite{specgraphfan}
\begin{equation}\label{eq:delta_def}
 \delta V = \{y \in V'\setminus V \,\,|\,\,\exists \,x \in V \,\,{\rm such}\,\,{\rm that}\,\, x\sim y \in E\},
 \end{equation}  
where  $x \sim y$ if $x$ is adjacent to the vertex $y,$ i.e. there is an edge in $E$ joining the vertex $x$ to $y.$ Here we assume that $V$ is a proper subset of $V'$ so that $\delta V$ is not empty. As in \cite{specgraphfan}, if $d_x$ is the degree of the vertex $x,$ we consider the (vertex) Laplacian $L: (V \cup \delta V) \times (V\cup\delta V) \rightarrow \mathbb{R}$ defined by
 \begin{equation}
   L(x,y) = \left\{
     \begin{array}{lr}
       d_x &  {\rm if}\,\, y = x\\
       -1 & {\rm if} \,\, y \sim x\\
       0 & {\rm otherwise}.
     \end{array}
   \right.
\end{equation} 
Note that in the following, by a slight abuse of notation, we will use the same symbol, $L,$ to denote the Laplacian operator, its kernel, and the corresponding matrix. In later sections we will employ a similar convention when discussing operators for the time-independent diffusion equation and their associated Green's functions.

To develop the time-independent diffusion equation on graphs we require suitable boundary conditions analogous to those arising in partial differential equations (PDEs). We say a function $u:V' \rightarrow \mathbb{R}$ satisfies a homogeneous {\it Dirichlet} boundary condition if its restriction to $\delta V$ is identically zero \cite{specgraphfan}. To obtain appropriate derivative-type boundary conditions we define the discrete analog of the normal derivative $\partial: \ell^2(V \cup \delta V) \rightarrow \ell^2(\delta V)$ by
\begin{equation}
\partial u(y) = \sum_{\substack{x \in V\\x \sim y}}[u(y)-u(x)].
\end{equation}
A function $u: (V \cup \delta V) \rightarrow \mathbb{R}$ satisfies a homogeneous {\it Neumann} boundary condition  \cite{specgraphfan} if $\partial u(x) = 0$ for all $x \in \delta V$ and satisfies a {\it Robin} boundary condition \cite{journ_func} if there exists a constant $t\ge 0$ such that
\begin{equation}\label{eq:rob_hom}
t\, u(x) + \partial u (x) = 0
\end{equation}
 for all $x \in \delta V.$ Note that choosing $t=0$ yields Neumann boundary conditions while letting $t \rightarrow \infty$ produces Dirichlet boundary conditions. Given a function $g: \delta V \rightarrow \mathbb{R}$ we can also define corresponding inhomogeneous boundary conditions
\begin{equation}\label{eq:rob}
t\, u(x) +\, \partial u (x) = g(x),\quad x \in \delta V
\end{equation}
which arise when sources or sinks are located on the boundary. For a given interior source $f$ and boundary source $g$ we define the constant absorption diffusion equation
\begin{equation}\label{eq:unpet_diff}
\left\{
     \begin{array}{lr}
\sum_{y \in V'} L(x,y) \,u(y) +\alpha_0 u(x) = f(x), \quad x \in V\\
t\, u(x) +  \partial u(x) = g(x), \quad x \in \delta V.
\end{array}
\right.
\end{equation}
 Here, in analogy with the physical problem of diffuse scattering, $\alpha_0$ is a strictly positive constant which represents the absorption of the medium. Note that $L$ is positive semidefinite \cite{journ_func,specgraphfan}. 

\subsection{Linear systems for finite boundary value problems}
In the case where $|V|$ and $|\delta V|$ are both finite the boundary value problem (\ref{eq:unpet_diff}) can be written as a linear system of equations for $u.$ We first index the vertices of $V$ by $1,\ldots, n = |V|$ and those of $\delta V$ by $n+1,\ldots, n+k$ where $k = |\delta V|.$ Next we construct the $(n+k) \times (n+k)$ matrix
\begin{equation}\label{eq:h_0_def}
H_0 =L+ \left(\begin{array}{cc} \alpha_0 {I}_{n\times n} &0_{n\times k} \\ {0}_{k\times n} & t I_{k\times k} \end{array}\right)
\end{equation}
where $I_{n\times n}$ and $I_{k\times k}$ are the $n\times n$ and $k\times k$ identity matrices, respectively,  and $0_{n\times k}$ is the $n \times k$ zero matrix. If we let $u = (u(x_1),\ldots,u(n_{n+k}))^*$ and $$\tilde{f} = (f(x_1),\ldots,f(x_n),g(x_{n+1}),\ldots, g(x_{n+k}))^*,$$ where $w^*$ denotes the conjugate transpose of $w,$ then we can rewrite the diffusion equation (\ref{eq:unpet_diff}) as
\begin{equation}\label{eq:unpet_mat_frm}
H_0\, {u} = \tilde{f}.
\end{equation}

Similarly, we obtain Dirichlet boundary conditions by replacing the matrix operator $H_0$ in (\ref{eq:h_0_def}) by the matrix
\begin{equation}
H_{0}^{{\rm D}} = \left(\begin{array}{cc} L(V;V)+\alpha_0 I_{n\times n} & L(V;\delta V)  \\ 0_{k \times n} & I_{k\times k} \end{array}\right)
\end{equation}
Here we have used the convention that given any two sets $A,B \subset V\cup\delta V,$ $L(A;B)$ is the submatrix of the Laplacian matrix, $L,$ obtained by taking the rows corresponding to the elements in $A$ and the columns corresponding to the elements in $B.$ We say the vector $u$ satisfies the diffusion equation with Dirichlet boundary conditions if
\begin{equation}\label{eq:dir_prob}
H_{0}^{{\rm D}} \, u  = \tilde{f}.
\end{equation}
Alternatively, one can obtain $u$ by noting that
\begin{equation}\label{eq:lim_dira}
u = \lim_{t \rightarrow \infty} u_t
\end{equation}
where $u_t$ satisfies the equation
\begin{equation}\label{eq:lim_dir}
H_0 u_t = \left(\begin{array}{l} f\\ t\, g \end{array}\right)
\end{equation}
and $H_0$ is the matrix operator corresponding to Robin boundary conditions depending on the parameter $t$ as in (\ref{eq:rob}).


It is clear by construction that $H_0$ is symmetric. As shown in the following proposition, under certain restrictions, the matrix $H_0$ is also positive definite and hence has a well-defined inverse. This is equivalent to the existence of a unique solution to the diffusion equation (\ref{eq:unpet_diff}). 

\begin{proposition}
\label{h_0_prop}
For all $t$ such that $0 \le t<\infty$ the smallest eigenvalue $\lambda_{\rm m}$ of $H_0$ satisfies
\begin{equation}
\lambda_{\rm m} \ge {\rm min}\{t, \alpha_0\}.
\end{equation}
It follows that the matrix $H_0$ is positive definite if $t >0$ then $H_0$ and positive semidefinite if $t=0,$ though later we will show that in the latter case $H_0$ is also positive definite provided that $\alpha_0 >0,$ see also \cite{ vert}.

\end{proposition}
\begin{proof}
The desired inequality follows immediately from an analysis of the variational formulation of the problem, as in \cite{journ_func,vert}. Alternatively, the result can also be shown by applying the Gerschgorin circle theorem to the operator $H_0.$

\end{proof}

\begin{proposition}
\label{h_0_prop_2}
Consider the diffusion equation (\ref{eq:unpet_diff}) on a connected graph $\Sigma$ with Neumann boundary conditions corresponding to $t= 0.$ The associated matrix operator $H_0$ is positive definite for all $\alpha_0>0$ and moreover 
\begin{equation}
\lambda_{\rm m} =\frac{|V|}{|V|+|\delta V|} \alpha_0+O(\alpha_0^2)
\end{equation}
as $\alpha_0 \rightarrow 0^+.$
\end{proposition}
\begin{proof}
The proof is by contradiction. Suppose $v$ is an eigenvector of $H_0$ with eigenvalue $0.$ Let $A$ be the matrix defined by
\begin{equation}
A_{i,j} = \left\{ \begin{array}{lr} 1, \quad &i=j, 1\le i\le n,\\ 0, \quad &{\rm otherwise}, \end{array}\right..
\end{equation} 
By construction it is clear that
\begin{equation}
H_0 = L+ \alpha_0 A
\end{equation}
and, that both $A$ is positive semidefinite. Note that
\begin{equation}
\begin{split}
0 &= v^*\, H_0\, v\\
&= v^*\,L \, v + \alpha_0 \, v^* \, A\, v. 
\end{split}
\end{equation}
Since $L$ and $A$ are positive semidefinite and $A$ is diagonal it is clear that $v$ is in the kernel of $A$ and is an eigenvector of $L$ with eigenvalue $0.$  Since $v \in {\rm ker}\, A$ it follows that its first $n$ entries must be identically zero. From \cite{specgraphfan} we observe that since $\Gamma$ is connected, the eigenvalue $0$ of $L$ has multiplicity one corresponding to the eigenvector $ (1,\ldots,1)^*.$ Thus, since $v$ is a scalar multiple of the all ones vector and its first $n$ entries are zero, it follows that $v$ is the zero vectorand hence cannot be an eigenvector of $H_0,$ completing the proof.

It follows immediately from the theory of asymptotic analysis of linear systems, see \cite{asympanal} for example, that the smallest eigenvalue of $L+\alpha_0 A$ is
\begin{equation}\label{eq:asymp_bnd}
\lambda_{\rm m}  = \alpha_0 \frac{v^* \, A v}{v^* v} + O(\alpha_0^2),
\end{equation}
from which the required result follows immediately.

\end{proof}
Plots of the minimum eigenvalue of $H_0$ as a function of $\alpha_0$ are shown for a path in Figure \ref{fig_path_a} and for a complete graph in Figure \ref{fig_comp_a}. As we can see, for small $\alpha_0$ the curve approaches the bound given in (\ref{eq:asymp_bnd}), which is shown in both plots for reference.
\begin{figure}
        \centering
        \begin{subfigure}[b]{0.45\textwidth}
                \centering
                \includegraphics[width=\textwidth]{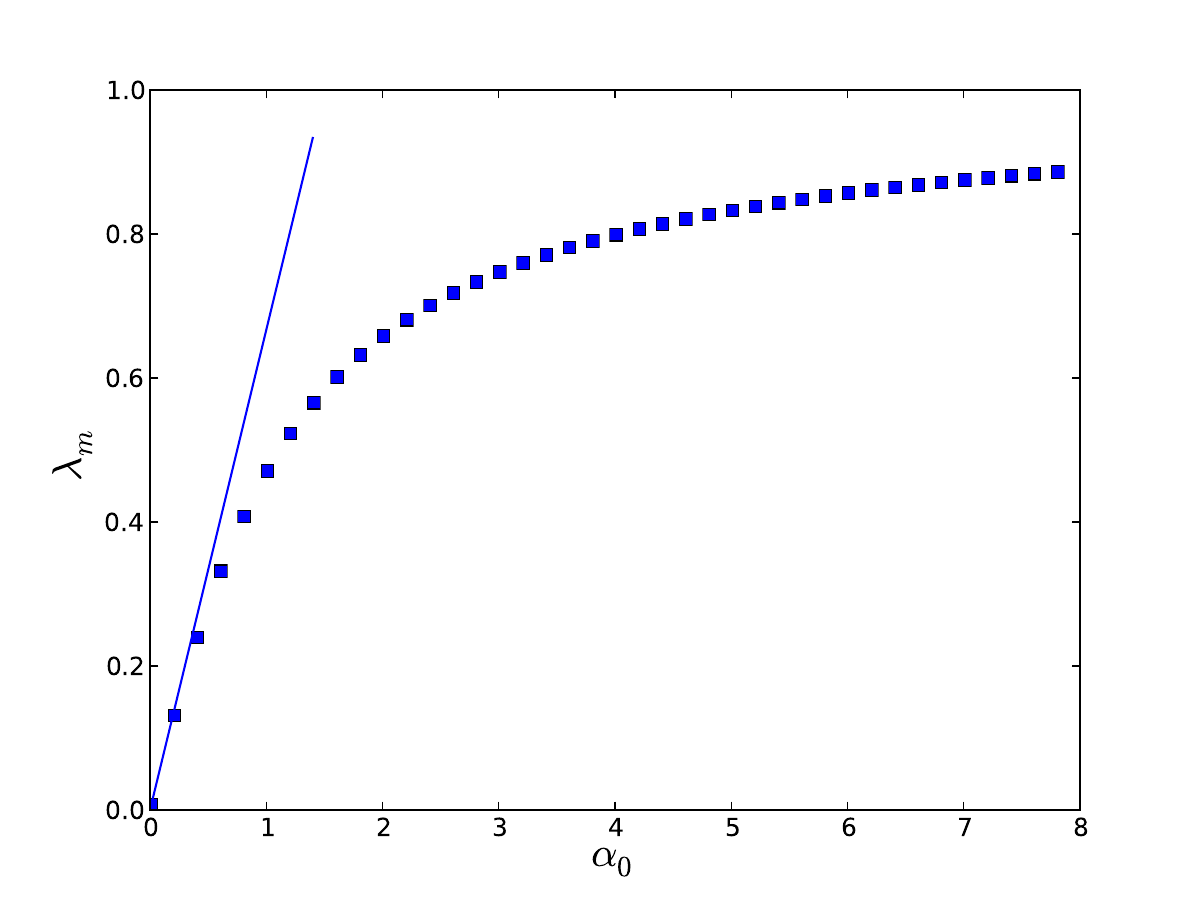}
                \caption{}
                \label{fig_path_a}
        \end{subfigure}%
        ~ 
        \begin{subfigure}[b]{0.45\textwidth}
                \centering
                \includegraphics[width=\textwidth]{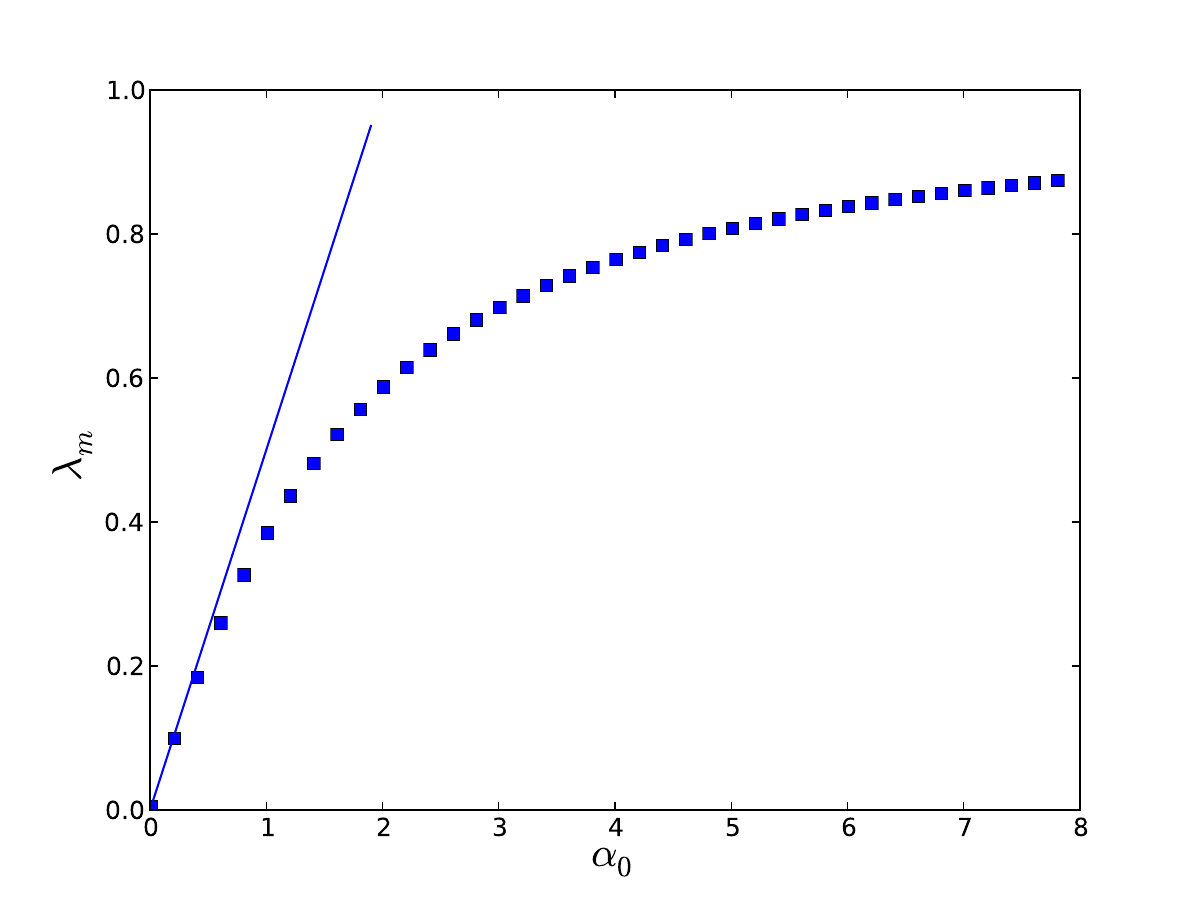}
                \caption{}
                \label{fig_comp_a}
        \end{subfigure}
        \caption{The minimum eigenvalue of the operator $H_0$ as a function of the absorption $\alpha_0$ for: a) a path of length $64$ with Neumann boundary conditions, and b) a complete graph on $64$ vertices and Neumann boundary conditions. For both plots the line corresponds to the bound in equation (\ref{eq:asymp_bnd}).}\label{fig:eig_p}
\end{figure}

\subsection{Spatially varying absorption}
When discussing diffusion problems in the continuous setting we often wish to consider media with spatially varying properties. A similar idea can be applied to graphs through a suitable modification of the graph diffusion problem (\ref{eq:unpet_diff}). Suppose the absorption at each vertex in $V$ is given by a non-negative function $\eta:V \rightarrow \mathbb{R}_{\ge0}.$ The resulting heterogeneous, or perturbed, diffusion equation is
 \begin{equation}\label{eq:pet_diff}
\left\{
     \begin{array}{lr}
\sum_{y \in V'} L(x,y) \,u(y) +\alpha_0\left[1+\eta(x)\right] u(x) = f(x), \quad &x \in V,\\
t\, u(x) +  \partial u(x) = g(x), \quad &x \in \delta V.
\end{array}
\right.
\end{equation}
Note that equation (\ref{eq:pet_diff}) also arises in the study of Schr\"{o}dinger problems on graphs where it can be interpreted as the Robin boundary value problem for the Schr\"{o}dinger operator with potential $q = \alpha_0(1+\eta).$ To write this as a linear system we let $D_\eta$ be the $(n+k)\times(n+k)$ matrix with entries
\begin{equation}\label{eq:def_eta}
(D_\eta)_{ij} = \left\{\begin{array}{lr}  \eta(x_i), \quad & i=j\le n, \\
0,\quad&{\rm otherwise.}
 \end{array} \right.
\end{equation}
It follows that $u$ solves the boundary value problem (\ref{eq:pet_diff}) if and only if it satisfies
\begin{equation}\label{eq:prelim_mat_eq}
[H_0 + \alpha_0 D_\eta] {u} = \tilde{f}.
\end{equation}
For convenience we define $H= H_0 +\alpha_0 D_\eta$ to be the matrix operator corresponding to the more general diffusion equation.

In many physical applications we are often interested in inhomogeneities confined to a region whose volume is significantly smaller than that of the whole domain. An analogous idea for diffuse scattering on graphs is to consider absorption functions, $\alpha(x),$ with small {\it support.} Given a function $\alpha(x)$ defined on a graph with vertex set $V,$ its {\it support} is the set of all vertices in $V$ for which the function $\alpha(x)$ is non-zero.

We also note that in some cases it is useful to consider graphs with no boundary. In this case we take $\delta V$ in (\ref{eq:pet_diff}) to be the empty set and consider the effects of sources placed in the interior, ie. the support of $\tilde{f}$ in (\ref{eq:prelim_mat_eq}) is contained in $V.$

\subsection{Green's functions for graphs}
Green's functions are a powerful tool for obtaining and analyzing solutions to PDEs such as the diffusion equation. When discussing similar equations on graphs, the analogous operator $G(x,y)$ is the inverse of $H$ \cite{specgraphfan}. Suppose the number of interior vertices of $\Sigma$ and the number of boundary vertices of $\Sigma$ are both finite. If we define ${G}_y = (G(x_1,y),G(x_2,y),\ldots, G(x_{n+k},y))^*$ and let $H^*$ denote the adjoint of the operator $H,$ then ${G}_y$ satisfies the linear system
\begin{equation}\label{eq:indiv_comp}
H^* {G}_y = \delta_y
\end{equation} 
where $\delta_y$ is the vector whose components are all zero except for the one corresponding to $y$ which is one. Observe that since $H= H_0 +\alpha_0 D_\eta,$ where $H_0$ and $\alpha_0$ are symmetric, it follows that $H$ is also symmetric, and so is its inverse $G.$ Using (\ref{eq:indiv_comp}) we see that if ${u}$ is a solution of (\ref{eq:pet_diff}) then
\begin{equation}
	\begin{aligned}
		{G}_y^* \,\tilde{f} &= {G}_y^*H {u}\\
		&= (H^* {G}_y)^*\, {u}\\
		&=\delta_y^*\, {u}\\
		&= u(y).\\
	\end{aligned}
\end{equation}
In this work the examples in which we are primarily interested are those for which $|V|$ and $|\delta V|$ are both finite, though when $|V|+|\delta V|$ is infinite Green's functions can also be defined, see \cite{journ_func,inf_urak2} for example.


\section{Born series}\label{sec_born}
We next discuss a useful perturbative method, called {\it Born series}, for constructing series solutions to (\ref{eq:pet_diff}) using the homogeneous Green's function.
\subsection{Construction}
Consider the matrix operator $H_0$ for the unperturbed diffusion equation (\ref{eq:unpet_diff}) and let $G_0$ be the matrix such that $G_0 H_0 = I.$ In particular we require the columns of $H_0$ to be linearly independent so that $H_0$ has a well-defined inverse. As in the previous section, we define the matrix operator $H$ for the perturbed problem (\ref{eq:pet_diff}) by
\begin{equation}\label{eq:defn_H}
H = H_0 + \alpha_0 D_\eta,
\end{equation}
where $D_\eta$ is once again the matrix defined in (\ref{eq:pet_diff}). Since $\eta \ge 0,$ $H^{-1}$ exists and satisfies
\begin{equation}
H^{-1} = \left( I + \alpha_0 G_0 D_\eta \right)^{-1} G_0
\end{equation}
and we can write a corresponding Neumann series
\begin{equation}\label{eq:Neumann_ser}
B = \left[ \sum_{n=0}^\infty (-1)^n  \alpha_0^{n} \left(G_0 D_\eta\right)^n\right] \, G_0,
\end{equation}
which, under suitable conditions on $G_0$ and $D_\eta,$ is equal to the inverse of $H$. In the context of scattering theory such an expansion is often called a {\it Born} series. Assuming the series in (\ref{eq:Neumann_ser}) converges to $H^{-1}$ it follows immediately that for any source vector $\tilde{f}$ the corresponding solution ${u}$ of the time-independent diffusion equation (\ref{eq:pet_diff}) is given by
\begin{equation}
{u} =  \left[ \sum_{n=0}^\infty (-1)^n (\alpha_0)^n \left(G_0 D_\eta\right)^n\right] \, G_0 \,\tilde{f}.
\end{equation}

\subsection{Convergence}\label{sec_con}
To show convergence of the Born series (\ref{eq:Neumann_ser}) with respect to a norm $\| \cdot \|$ it is sufficient to show that the induced operator norm of $B,$ denoted by $\|B\|,$ is bounded, as shown in the following theorem.

\begin{theorem}
\label{born_prop_1}
The series
\begin{equation}
B = \left[ \sum_{n=0}^\infty (-1)^n  \alpha_0^{n} \left(G_0 D_\eta\right)^n\right] \, G_0
\end{equation}
converges to the Green's function of the perturbed problem (\ref{eq:pet_diff}) if $ \alpha_0 \,\|G_0\|\cdot \|D_\eta\|<1.$ Moreover, if $B_N,$ the truncated operator formed by taking the first $N+1$ terms of the Born series, we have the following estimate of the error
\begin{equation}
\|B-B_N\| \le \|G_0\|^2 \frac{\alpha_0^N \|G_0\|^N \|\eta\|_\infty}{1-\alpha_0 \|G_0\|\,\|\eta\|_\infty}
\end{equation} 
\end{theorem}

\begin{proof}
The proof is an immediate consequence of the existing theory of {\it Neumann series}, see \cite{neumann} for example.
\end{proof}

In particular we see from the previous theorem that approximating the Green's function by a truncated Born series is more accurate when $\|D_\eta\| \, \|G_0\|^{-1}\, \alpha_0^{-1}$ are small, sometimes called the {\it weak scattering limit} \cite{weak_scat}. We can also obtain tighter bounds if additional information about the structure of the absorption matrix $D_\eta$ is used. In particular it is natural to assume that the matrix  $D_\eta$ has few  non-zero diagonal entries. This is analogous to the physical situation where the spatial support of the scatterers is much smaller than the total volume.

\begin{proposition}
\label{born_prop_2}
Suppose $\eta$ has support $\Lambda \subseteq V$ and let $I_\Lambda$ be the restriction of the identity matrix to the support of $\eta.$ Further define $G_{0,\Lambda} = I_\Lambda G_0 I_\Lambda$ and let $\eta_{\rm max} =  \sup_{x \in \Lambda} \eta(x). $ The series
\begin{equation}
B = \left[ \sum_{n=0}^\infty (-1)^n  \alpha_0^{n} \left(G_0 D_\eta\right)^n\right] \, G_0
\end{equation}
converges to the Green's function of the perturbed problem (\ref{eq:pet_diff}) if $ \eta_{\rm max} \alpha_0 \,\|G_{0,\Lambda}\| <1.$ Moreover, the truncation error associated with taking the first $N+1$ terms of the Born series,
\begin{equation}
B_N =\left[ \sum_{n=0}^N (-1)^n  \alpha_0^{n} \left(G_0 D_\eta\right)^n\right] \, G_0,
\end{equation}
is $O\left(  \alpha_0^{N} \,||G_{0,\Lambda}||^N \cdot \eta_{\rm max}^N\right)$ as $N\rightarrow \infty.$
\end{proposition}

\begin{proof}
Since $D_\eta$ is a diagonal matrix it follows that $\|D_\eta\| = \eta_{\rm max} = \sup_{x \in V} |\eta(x)|.$ Let $\Lambda$ be the support of $\eta$ and let $I_\Lambda$ be the restriction of the identity matrix to the support of $\eta.$ In particular, $I_\Lambda$ is the diagonal matrix $I_{\Lambda}(x,y) = \delta_{x,y} \chi_{\{x \in \Lambda\}},$ where $\chi_A$ denotes the characteristic function of the set $A.$  Note that $D_\eta = I_\Lambda D_\eta = D_\eta I_\Lambda$ and thus if we define $G_{0,\Lambda} = I_\Lambda G_0 I_\Lambda$ and let $n >1,$ then $(G_0 D_\eta)^n = G_0 D_\eta (I_\Lambda G_0 I_\Lambda D_\eta)^{n-1} = G_0 D_\eta (G_{0,\Lambda} D_\eta)^{n-1}.$ Defining the truncated operator $B_N = \sum_{n=0}^{N} (-1)^k \alpha_0 \left(G_0 D_\eta\right)^n G_0,$ we note that
\begin{equation}
B_N = G_0 +G_0 D_\eta\sum_{n=1}^{N-1} (-1)^k\alpha_0^n \left( G_{0,\Lambda}D_\eta\right)^{n-1} G_0.
\end{equation}
The result now follows immediately from the theory of Neumann series.
\end{proof}


\section{Examples}\label{sec_examp}
Having developed the theory of Born series in the previous section, provided the perturbations to the absorption are sufficiently small, we can now apply this method to approximate Green's functions for which the background Green's function is known. A non-exhaustive list of families graphs for which the background Green's function is known is given in Table \ref{table:background_greens_functions}.

\afterpage{%
    \clearpage
    \thispagestyle{empty}
    \begin{landscape}%
\begin{longtable}{ | c | m{3cm} | m{9cm} |  c|}
    \hline
    Name & Figure & Background Green's function& Reference\\[5pt] \hline &&&\\
    Path&
    \begin{minipage}{.2\textwidth}
    \begin{center}
      \includegraphics[width=\linewidth]{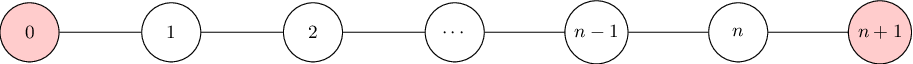}
      \end{center}
    \end{minipage}
    &
    $G(i,j) =\frac{(a r^i - a^{-1} r^{-i})(a r^{n+1-j} - a^{-1} r^{-(n+1-j)})}{\left( r- \frac{1}{r}\right)\left(a^2 r^{n+1} - a^{-2} r^{-(n+1)}\right)},$ \scriptsize $0 \le i \le j \le n+1,$ $r+1/r = 2+ \alpha_0,$ and $a =\left[ 1+\frac{ (r^2-1)}{r[1+t-r]}\right]^{1/2}$&\cite{eig_path,discgreen}
    \\[10pt] \hline &&& \\
     Cycle&
    \begin{minipage}{.2\textwidth}
    \begin{center}
      \includegraphics[width = 0.6\textwidth]{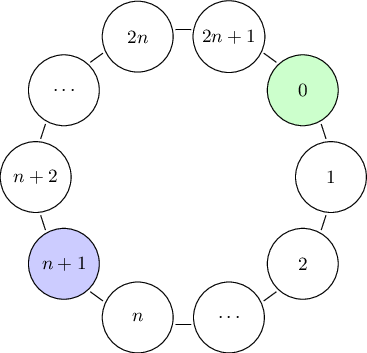} 
      \end{center}
    \end{minipage}
    &
$$G(i,j) =  \frac{r^{n+1-|i-j|_{\rm min}} + r^{-(n+1-|i-j|_{\rm min})}}{\left( r - \frac{1}{r} \right) \left(r^{n+1} - r^{-(n+1)}\right)},$$ \newline \scriptsize for all $0 \le i,j \le 2n+1,$ where $|i-j|_{\rm min} = \min\{ |i-j|, \, 2n+2-|i-j| \}.$&\cite{eig_path,ellis}
    \\[10pt] \hline &&&\\
   M\"{o}bius ladder&
    \begin{minipage}{.2\textwidth}
    \begin{center}
      \includegraphics[width = 0.6\textwidth]{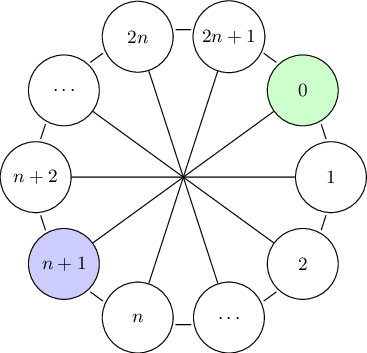} 
      \end{center}
    \end{minipage}
    &
    \begin{minipage}{.5\textwidth}
    \scriptsize
\begin{equation}\nonumber
G(i,j) = \left\{ \begin{array}{lr} g_1(|i-j|_{\rm min}) + g_2(|i-j|_{\rm min}), \, & |i-j|\le \frac{n+1}{2}+1  \\
 g_1(|i-j|_{\rm min})-g_2(|i-j|_{\rm min}),\, & |i-j| > \frac{n+1}{2}\end{array} \right.
\end{equation}
\scriptsize
for all $0 \le i,j \le 2n+1,$ where $|i-j|_{\rm min} = \min\{ |i-j|, \, 2n+2-|i-j|,\, \left|\,|i-j| - (n+1) \right| \},$
\begin{equation}\nonumber
g_k(s) = \frac{\left(a_k r_k^{\frac{n+1}{2}}-\frac{r_k^{-\frac{n+1}{2}}}{a_k}\right)\left(a_k r_k^{\frac{n+1}{2}-s}-\frac{r_k^{-[\frac{n+1}{2}-s]}}{a_k}\right) }{(r_k-r_k^{-1}) (a_k^2 r_k^{n+1}-a_k^{-2} r_k^{-(n+1)})},
\end{equation} 
$k\in\{1,2\},$ $r_k$ satisfies $r_k+1/r_k = 2k+\alpha_0,$ \\$a_1 = \left[1+\frac{r^2_1-1}{r_1(1+\frac{\alpha_0}{2}-r_1)} \right]^\frac{1}{2},$ $a_2 = 1.$
\vspace{0.2 cm}

\end{minipage}
   & Appendix A\\[15pt] \hline &&&\\
        Complete graph&
    \begin{minipage}{.15\textwidth}
    \begin{center}
      \includegraphics[width=\linewidth]{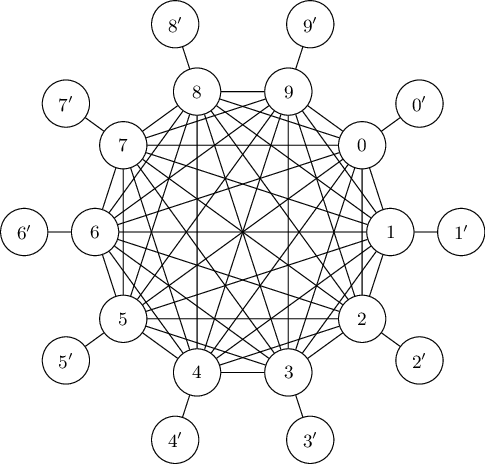}
      \end{center}
    \end{minipage}
    &
    \begin{minipage}{0.5\textwidth}
\begin{equation}\nonumber
   G(x,y) = \left\{
     \begin{array}{lr}
       \frac{\sigma}{(\sigma-1)(\sigma-1+d)} &{\rm if}\,\, x =y \in V,\\
       
       \frac{1}{(\sigma-1)(\sigma-1+d)} &{\rm if}\,\, x \neq y,\, x,y \in V,\\
       \frac{\gamma^2 \sigma}{(\sigma-1)(\sigma-1+d)}+\gamma &{\rm if}\,\, x = y\in \delta V,\\
     \end{array}
   \right.
\end{equation} 
\scriptsize
where $\gamma = 1/(1+t)$ and $\sigma = 2+ \alpha_0 - \gamma.$
\vspace{0.2 cm}

     \end{minipage} & Appendix A

    \\[10pt] \hline &&& \\
    Infinite plane&
    \begin{minipage}{.2\textwidth}
    \begin{center}
      \includegraphics[width=0.4\linewidth]{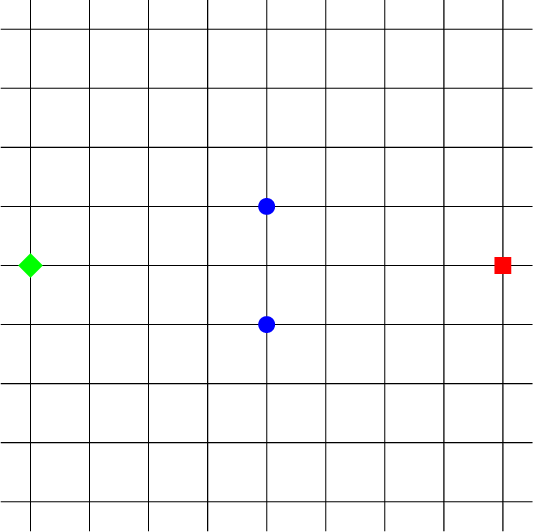}
      \end{center}
    \end{minipage}
    &
    \begin{minipage}{0.5\textwidth}
    \tiny
    \begin{equation}\nonumber
G((m_1,n_1),(m_2,n_2))= \frac{1}{2\pi} \int_0^\pi \frac{\cos\left( d_- \,v\right)\, (\cos(v))^{d_+}}{(a+\sqrt{a^2-\cos(v)})^{d_+} \sqrt{a^2 - \cos(v)}}\,{\rm d}v
\end{equation}
\scriptsize
where $a = 1+\alpha_0/4,$ and $d_\pm = |m_2-m_1| \pm |n_2-n_1|.$
\vspace{ 0.1 cm}
     \end{minipage}& Appendix A
    \\[10pt] \hline
  \caption{Summary of Green's function results}\label{table:background_greens_functions}
\end{longtable}
    \end{landscape}
    \clearpage
}
\normalsize
\clearpage

\section{Representation theory and the background Green's function}\label{sec_rep_theory}
In the previous section we obtained the background Green's function for a variety of examples by solving corresponding recurrence relations. In every example, excluding the finite path, clearly-visible symmetries were employed in an intuitive way to reduce the problem to a more tractable set of equations. This connection between symmetry and PDE analogues on graphs can be formalized using the language of representation theory. 

\subsection{Cayley graphs of finite abelian groups}

As a particularly important special case we first consider Cayley graphs of abelian groups. In particular, let $G$ be a finite abelian group and $S$ be a symmetric subset of the elements of $G.$ Recall that $S$ is a symmetric subset of a group $G$ if $g \in S$ implies $g^{-1} \in S.$ This condition is required to ensure that the resulting graph is undirected and the associated Laplacian operator is symmetric. We can then define the the Cayley graph $X(G,S)$ to be the graph whose vertices are indexed by the elements of $G$ with edge set \cite{terras}
\begin{equation}
E = \{(g,h) \in G\times G \, |\, gh^{-1} \in S \}.
\end{equation}
Looking at the examples considered in the previous section we note that the loop, M\"{o}bius ladder, and complete graph are all Cayley graphs with group $G=\mathbb{Z}/ n \mathbb{Z},$ for some $n,$ and $S = \{-1,1\},$ $\{-1,1,-n/2,n/2\}$ and $\{-n+1,-n+2,\ldots,-1,1,\ldots, n-2,n-1\},$ respectively. The infinite path is the Cayley graph of the free group on 2 generators and the Bethe lattice with coordination number $k$ is the Cayley graph of the free group on $k$ generators. Finally, the two-dimensional lattice is the Cayley graph with group $\mathbb{Z} \times \mathbb{Z}$ and generator set $S = \{(-1,0),(0,-1), (1,0),(0,1)\}.$

For Cayley graphs the combinatorial Laplacian can be compactly expressed using the convolution operator $\ast: \ell^2(G) \times \ell^2(G) \rightarrow \ell^2(G)$ defined by
\begin{equation}
(f\ast g)(x) = \sum_{y \in G} f(y) \,g(y^{-1}x).
\end{equation}
It is clear that the adjacency matrix $A$ for $X(G,S)$ is given by \cite{terras}
\begin{equation}
A (f)(x) = (\delta_S \ast f)(x)
\end{equation}
where $\delta_S$ is the characteristic function on the set $S$ and hence if $k = |S|$ then
\begin{equation}\label{eq:def_delt_lap}
L(f)(x) =( k\, I - A)(f)(x) = k f(x) - (\delta_S \ast f)(x).
\end{equation}

In order to diagonalize this operator using Fourier analysis, we next define the {\it dual group}
\begin{equation}\label{eq:ab_hom}
\hat{G} = {\rm Hom}(G,\mathbb{T})
\end{equation}
where $\mathbb{T}$ is the multiplicative group of complex numbers with modulus one. If $\chi \in \hat{G}$ then we call $\chi$ a {\it character}. If $G$ is a finite abelian group then it is self-dual \cite{serre}, and hence $G$ is isomorphic to $\hat{G}.$ We then have the following proposition, proved in \cite{terras}.

\begin{proposition}\label{prop_terras}
If $h \in \ell^2(G)$ then the eigenvectors of the corresponding convolution operator are equal to the characters of $G.$ In particular, if $\chi \in \hat{G}$ then
\begin{equation}
(h \ast \chi)(x) = \hat{h}(\chi)\, \chi(x)
\end{equation} 
for all $x \in G$ and where $\hat{h}(\chi) = \sum_{x \in G} h(x) \overline{\chi(x)}.$
\end{proposition}

The following corollaries are immediate consequences of Proposition \ref{prop_terras}.
\begin{corollary}
The characters of $G$ are the eigenvectors of $L$ with corresponding eigenvalues
\begin{equation}
\lambda_\chi = k-\sum_{s \in S} \overline{\chi(s)}.
\end{equation}
\end{corollary}

\begin{corollary}
The Green's function for the uniform diffusion equation is
\begin{equation}\label{eq:rep_green}
G (f)(x) = \sum_{\chi \in \hat{G}} \,\,\sum_{y \in G} \frac{1}{\lambda_\chi+\alpha_0} f(y) \overline{\chi(y)}\chi(x).
\end{equation}
\end{corollary}
\begin{proof}
We begin by observing that since the eigenfunctions of $L$ are the characters of $G$ through an appropriate change of basis we may diagonalize $L.$ If we denote the elements of $G$ by $x_1,\ldots, x_k$ and the characters of $G$ by $\chi_1,\ldots,\chi_k$ then we can form the corresponding $k \times k$ matrix defined by
\begin{equation}
X_{i,j} = \chi_j(x_i).
\end{equation}
It follows from the above results that the matrix representation of $L$ can be written as
\begin{equation}
L = X \begin{pmatrix}\lambda_{\chi_1}\\&\lambda_{\chi_2}\\&&\ddots\\&&&\lambda_{\chi_{k-1}}\\&&&&\lambda_{\chi_k} \end{pmatrix} X^{\dag},
\end{equation}
where $X^\dag$ denotes the conjugate transpose of the matrix $X.$ Since $XX^\dag = I,$ the $k \times k$ identity matrix, it follows that
\begin{equation}
\begin{split}
H_0 &= L + \alpha_0 \,I \\
&= X \begin{pmatrix}\lambda_{\chi_1}+\alpha_0\\&\lambda_{\chi_2}+\alpha_0\\&&\ddots\\&&&\lambda_{\chi_{k-1}}+\alpha_0\\&&&&\lambda_{\chi_k}+\alpha_0 \end{pmatrix} X^{\dag},
\end{split}
\end{equation}
from which the formula (\ref{eq:rep_green}) follows immediately, using the fact that $G = H_0^{-1}.$
\end{proof}

\subsection{Cayley graphs of finite groups}

The results of the previous section can be extended to non-abelian groups in a natural way. As before, we use the fact that the operator $H_0$ can be written as a convolution operator on $\ell^2(G)$ to find a spectral decomposition of its Fourier transform. Applying the inverse Fourier transform yields a complete set of eigenvectors and eigenvalues of $H_0$ from which it is straightforward to obtain an expression for the background Green's function $G_0$ of the corresponding Cayley graph.

Before presenting the main results we first introduce some basic definitions and results associated with Fourier analysis on finite non-abelian groups (for a more complete description see \cite{terras}, for example). As in \cite{serre}, let $\rho$ be a homomorphism from the group $G$ to the automorphism group of $V,$ a $k$-dimensional vector space over $\mathbb{C}.$ Such a map $\rho$ is called a $k$-{\it dimensional representation} of $G$ and is said to be {\it irreducible} if the only subspaces of $V$ which are invariant under $\rho(g)$ for all $g \in G$ are $\bf{0}$ and $V.$  We say that two representations $\rho:G \rightarrow {\rm GL}(V)$ and $\tau: G \rightarrow {\rm GL}(W)$ are {\it equivalent} \cite{serre} if there exists an isomorphism $f:V \rightarrow W$ such that $\tau(g) = f \circ \rho(g) \circ f^{-1}.$

Given a representation $\rho: G \rightarrow{\rm GL}(V)$ of $G,$ let $d_\rho={\rm dim}(V)$ denote its degree. If $f \in \ell^1(G)$ then its {\it Fourier transform} is the map $\mathcal{F}[f]:\rho \rightarrow \mathbb{C}^{d_\rho}\times \mathbb{C}^{d_\rho}$ defined by \cite{terras}
\begin{equation}
\mathcal{F}[f](\rho) =\sum_{g \in G} f\,(g) \rho(g).
\end{equation}
Observe that the Fourier transform of a function $f$ at a representation $\rho$ will, in general, be matrix-valued and is called the {\it Fourier coefficient} matrix of $f$ at $\rho.$ If two representations $\rho_1$ and $\rho_2$ are equivalent then it is straightfoward to show that the corresponding Fourier coefficient matrices are similar and hence the Fourier transform of $f$ is completely determined by its value on a maximal set of inequivalent irreducible representations, called the {\it dual} and denoted by $\hat{G}.$ Note that if $G$ is abelian then its irreducible representations must be of degree one and this definition of $\hat{G}$ is equiavalent to the one given in (\ref{eq:ab_hom}).

Given a function $h: \rho \in \hat{G} \rightarrow \mathbb{C}^{d_\rho} \times \mathbb{C}^{d_\rho}$ we can define its inverse Fourier transform by the following expression \cite{terras}
\begin{equation}
\check{h} = \mathcal{F}^{-1}[h](g) =\frac{1}{|G|} \sum_{\rho \in \hat{G}} d_{\rho} {\rm Tr}\left(\rho(g^{-1}) h(\rho)\right).
\end{equation}
The proof that $\mathcal{F}^{-1} \mathcal{F}$ is the identity operator on $\ell^2(G)$ and is independent of the choice of the elements in $\hat{G},$ provided they form a maximal set of irreducible inequivalent representations can be found in \cite{terras}.

From the definitions given above it is straightforward to prove the following result \cite{terras} which will be useful in decomposing the operator $H_0.$
\begin{proposition}\label{non_ab_conv}
Consider the convolution operator $*:\ell^1(G) \times \ell^1(G) \rightarrow \ell^1(G)$ defined by
\begin{equation}
f *h (g)  = \sum_{r \in G} f(r^{-1}) h(rg).
\end{equation}
If $\hat{f} = \mathcal{F}[f]$ and $\hat{h} = \mathcal{F}(h)$ then
\begin{equation}
\mathcal{F}[f*h](\rho) = \hat{f}(\rho) \hat{h}(\rho).
\end{equation}
\end{proposition}

We can now employ the theory developed above to analyze the spectrum of the Cayley graph $X(G,S)$ with vertices once again indexed by the elements of $G$ and edge generating set $S.$ We begin by observing that the adjacency operator $A$ for $X(G,S)$ can be written as
\begin{equation}
A[f](g) = \chi_S * f (g), \quad \forall g \in G
\end{equation}
where $\chi_S$ is the characteristic function of $S.$ It follows immediately that if $e \in G$ is the identity element then
\begin{equation}
H_0 [f](g) = \left( (|S|+\alpha_0)\chi_e-\chi_S \right) *f(g)
\end{equation}
and thus from Proposition \ref{non_ab_conv} that
\begin{equation}\label{eq:non_ab_eig}
\hat{H_0}[\hat{f}] (\rho) = \left(|S|+\alpha_0\right) \hat{f}(\rho)- \sum_{g \in S} \rho(g) \hat{f}(\rho).	
\end{equation}
The problem then becomes to diagonalize the operator $\hat{H}_0$ in Fourier space by finding suitable eigenfunctions of (\ref{eq:non_ab_eig}). Using the properties of the Fourier transform outlined above we can then find the corresponding eigenfunctions of $H_0$ in $\ell^2(G).$

For ease of exposition, let $M(\rho) = \sum_{g \in S} \rho(g)$ in which case we have the following useful lemma.
\begin{lemma}
If $M(\rho) = \sum_{g \in S} \rho(g)$ where $\rho$ is a representation of a finite group, $G,$ then $M(\rho)$ is diagonalizable.
\end{lemma}
\begin{proof}
We begin by noting that since $G$ is finite, $\rho$ is equivalent to a representation $\rho'$ such that $\rho'(g)$ is unitary for all $g \in G.$ In particular, there exists a $d_\rho \times d_\rho$ matrix $B$ such that
\begin{equation}
M(\rho) = B \, M(\rho') \, B^{-1}.
\end{equation}
Since $\rho'(g)$ is unitary observe that $\rho'(g^{-1}) = \rho'(g)^{-1} = \rho'(g)^*,$ where $*$ once again denotes the adjoint of a matrix. Furthermore, we note that since $S$ is symmetric, if $g \in S$ then $g^{-1} \in S$ so that
\begin{equation}
M(\rho') = \frac{1}{2} \sum_{g \in S} \left[ \rho'(g) + \rho'(g^{-1})\right].
\end{equation}
 It follows immediately that
\begin{equation}
\begin{split}
M(\rho')^* &= \left[\sum_{g \in S} \rho'(g)\right]^*\\
&= \frac{1}{2} \sum_{g \in S} \left[ \rho'(g)+\rho'(g^{-1})\right]^*\\
&= \frac{1}{2} \sum_{g \in S} \left[ \rho'(g) + \rho'(g)^*\right]^*\\
&= M(\rho'),
\end{split}
\end{equation}
and so $M(\rho')$ is Hermitian. Since $M(\rho)$ is similar to $M(\rho')$ it follows that it too is diagonalizable.
\end{proof}

Fixing a maximal set of inequivalent irreducible representations $\hat{G} = \{ \rho_1,\ldots, \rho_L\},$ let $v^i_j$ be the $j$th eigenvector of $M(\rho_i)$ with eigenvalue $\nu^i_j.$ Furthermore, let $H^i_{jk}$ be the $d_{\rho_i} \times d_{\rho_i}$ zero matrix with the $k$th column replaced by $v^i_j.$ We then define the function $f^i_{jk}: \rho \in \hat{G} \rightarrow \mathbb{C}^{d_\rho} \times \mathbb{C}^{d_\rho}$ by
\begin{equation}\label{eq:four_eig}
f^i_{jk}(\rho) = \delta_{\rho_i}(\rho) H^i_{jk}
\end{equation}
where $\delta_{\rho_i}(\rho)$ is $1$ if $\rho = \rho_i$ and zero otherwise. We remark that the function $f$ has only been defined on the set of representations $\hat{G}$ though it has a natural, and unique, extension, $\tilde{f},$ to all representations of $G$ by requiring the following two conditions hold:
\begin{enumerate}[i)]
\item if $\rho$ and $\rho'$ are equivalent representations such that $\rho = B\, \rho'(g)\, B^{-1}$ for all $g \in G$ then $\tilde{f}^i_{jk}(\rho) = B^{-1} f^i_{jk}(\rho') B,$ and
\item if $\rho = \rho_1 \oplus \rho_2$ then $\tilde{f}^i_{jk} (\rho) = f^i_{jk}(\rho_1) \oplus f^{i}_{jk}(\rho_2).$
\end{enumerate}
The following proposition is an immediate consequence of the previous definitions.
\begin{proposition}\label{prp:non_ab_four_eig}
The function $f^i_{jk} (\rho)$ defined in (\ref{eq:four_eig}) is an eigenfunction of the operator $\hat{H}_0$ with corresponding eigenvalue
\begin{equation}\label{eq:non_ab_eigval}
\lambda^{i}_{jk} = |S|+\alpha_0-\nu^{i}_j.  
\end{equation}
for all $i =1,\ldots,L$ and $1 \le j,k \le d_{\rho_i}.$
\end{proposition}

Notice that Proposition \ref{prp:non_ab_four_eig} tells us that given the irreducible representations of $G$ we can reduce the problem of finding the eigenfunctions and eigenvalues of the operator $\hat{H}_0$ to that of finding the eigenvectors and eigenvalues of the matrices $\{ M(\rho_i)\}_{i=1}^L.$ To find the corresponding eigenfunctions of the operator $H_0$ we observe that
\begin{equation}
H_0[f](g) = \mathcal{F}^{-1}\left[H_0 \mathcal{F}[f]\right](g)
\end{equation}
from which it follows that $H_0$ has eigenfunctions $u^i_{jk} = \mathcal{F}^{-1} {f}^{i}_{jk}$ with eigenvalues $\lambda^{i}_{jk}.$ Proceeding in this way will generate $\sum_{\rho \in \hat{G}} d_{\rho}^2$ distinct eigenfunctions of $H_0.$ However, we know that \cite{serre}
\begin{equation}
n = |G| = \sum_{\rho \in \hat{G}} d_{\rho}^2,
\end{equation}
and so the above procedure will produce a complete set of eigenfunctions for the operator ${H}_0.$ Since $H_0$ is Hermitian we can use Gram-Schmidt orthogonalization to produce a complete orthonormal set of eigenvectors $\phi^i_{jk},$ with eigenvalues once again given by $\lambda^i_{jk},$ from which it follows that
\begin{equation}\label{irrep_g_func}
G_0 =\sum_{i=1}^L\,\, \sum_{1 \le j,k \le d_{\rho_i}}\frac{1}{\lambda^{i}_{jk}} \phi^i_{jk} {\phi^i_{jk}}^*.
\end{equation}

As an illustration of this procedure we now consider the background Green's function for the permutohedron of order $4,$ shown in Figure \ref{fig_perm}, though the following analysis generalizes to permutohedra of arbitrary order. 

\begin{figure}[h!]
  \centering
    \includegraphics[width=0.5\textwidth]{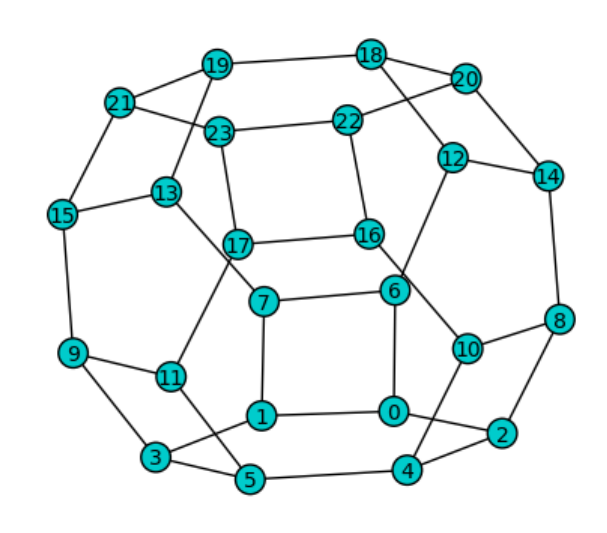}
      \caption{The permutahedron of order $4.$}\label{fig_perm}
\end{figure}

We begin by noting that the permutohedron of order $n$ is isomorphic to the Cayley graph $X(S_n,S),$ where $S_n$ is the symmetric group on $n$ letters and $S$ is the symmetric set of generators consisting of all transpositions which interchange neighbouring elements \cite{angel}. For each irreducible representation $\rho$ of $S_n$ we can construct the matrix $M(\rho),$ given by
\begin{equation}
M(\rho) = \sum_{g\in S}\rho(g).
\end{equation}
Next, for each non-equivalent irreducible representation $\rho,$ we diagonalize the matrix  $M(\rho)$ and form the eigenvectors of $\hat{H}_0$ using (\ref{eq:four_eig}). Taking the inverse Fourier transform of each of these eigenvectors yields eigenvectors of the original operator $H_0,$ with corresponding eigenvalues (\ref{eq:non_ab_eigval}). After normalizing the eigenvectors we construct the background Green's function using (\ref{irrep_g_func}). The matrix given by this procedure, plotted in Figure \ref{fig_gf_permu}, agrees to within machine precision with the inverse of $H_0$ calculated numerically.

\begin{figure}[h!]
  \centering
    \includegraphics[width=0.5\textwidth]{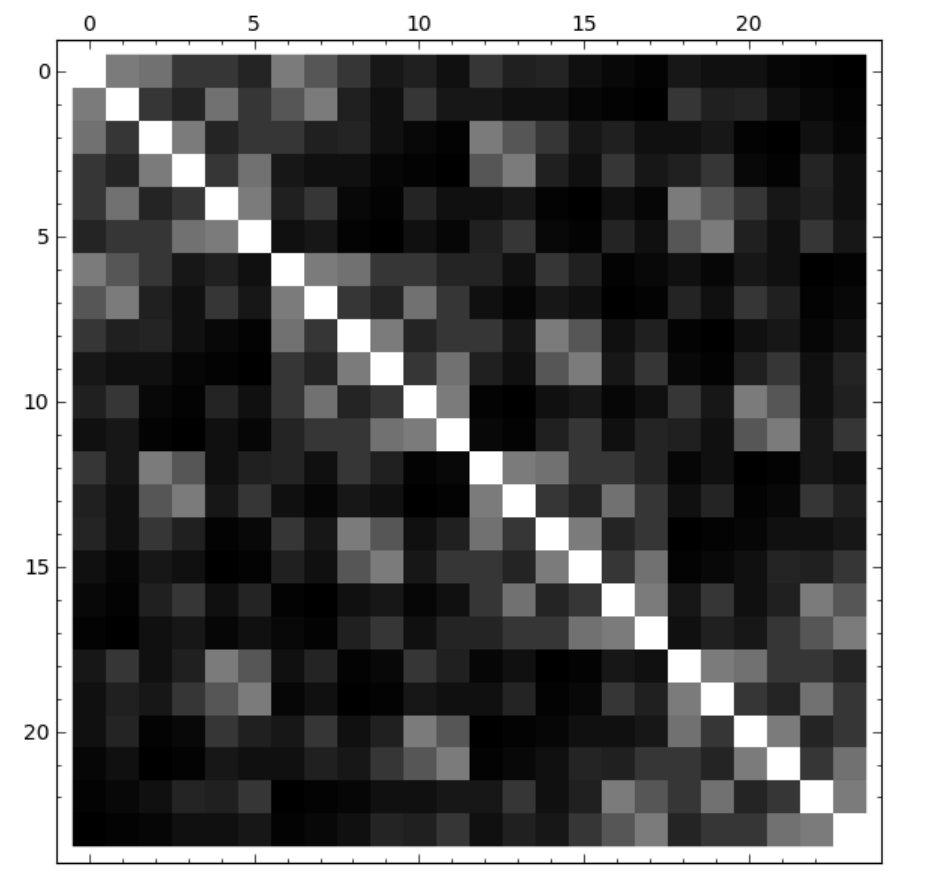}
      \caption{The background Green's function for the permutohedron of order $4,$ with $\alpha_0 = 0.1.$}\label{fig_gf_permu}
\end{figure}
\clearpage

\section{Numerical experiments}\label{sec_num_exp}
In this section we demonstrate the use of the Born series for two illustrative examples to approximate Green's functions when the absorption is a small perturbation of a constant background value $\alpha_0,$ comparing the convergence we observe with the bounds obtained in Section \ref{sec_born}.

\subsection{Inhomogeneous absorption on a path}
Using the background Green's function for the path given in Table \ref{table:background_greens_functions}, we can solve the diffusion equation (\ref{eq:pet_diff}) on a path provided the absorption coefficients $\eta(x)$ are sufficiently small. Let $u_N = B_N \tilde{f}$ where $\tilde{f}$ is the source vector and $B_N$ is the truncated Born series matrix operator. Figure \ref{fig_path_error} gives the numerical results obtained when using the particular $\eta$ given in Figure \ref{fig_path_eta} with a unit source located at the left boundary vertex and $t =1/2$. As predicted the error decays exponentially as $N \rightarrow \infty$ if $\eta_{\rm max}$ is less than a cut-off value, which is approximately $1.15.$ The comparison between the empirically determined cut-off for $\eta_{\rm max}$ and the upper bounds given by Section \ref{sec_con} is summarized in Table \ref{tab_path_cut_off}.

\begin{figure}[h!]
  \centering
    \includegraphics[width=0.5\textwidth]{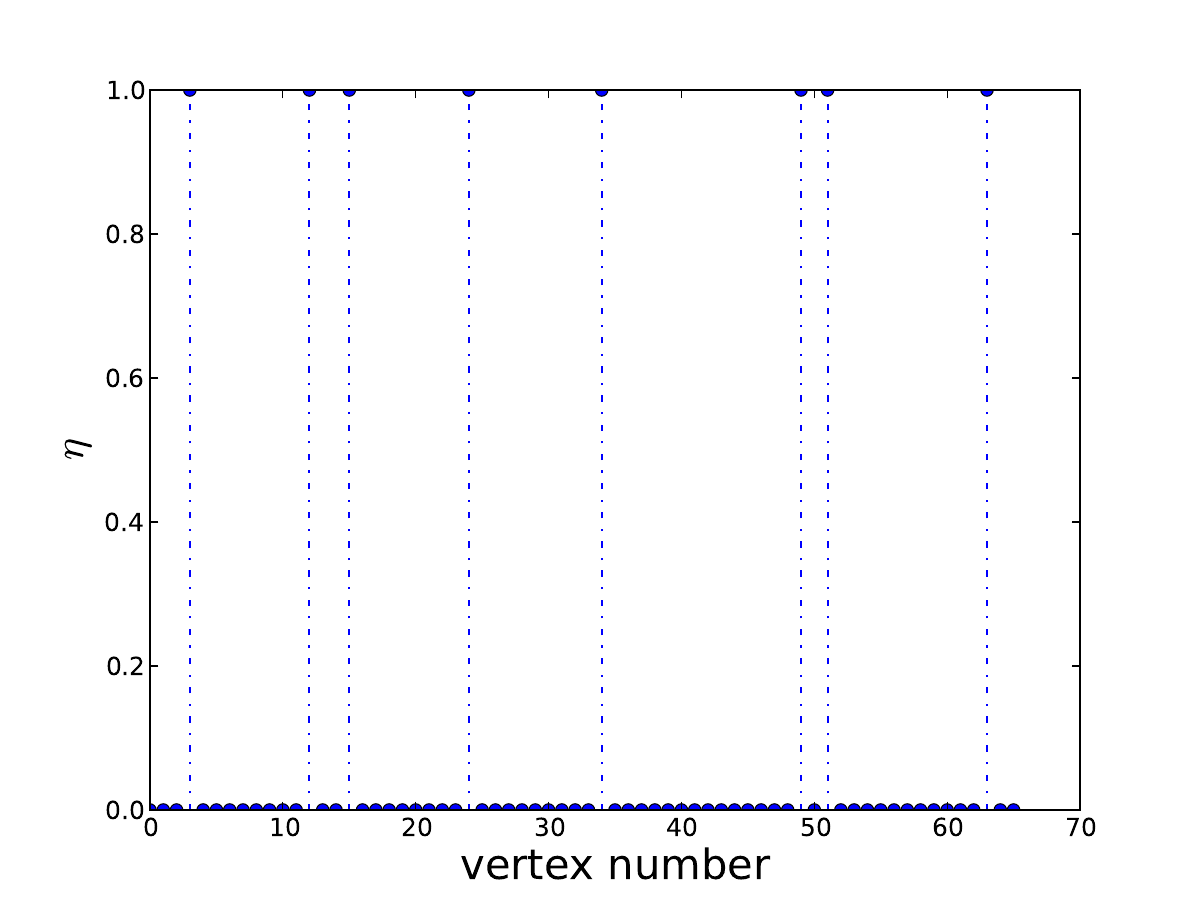}
      \caption{ The absorption vector $\eta$ used for the example of constructing a Green's function for the heterogeneous diffusion equation (\ref{eq:unpet_diff}) via Born series. The support of $\eta$ is chosen to be a random subset of the interior vertices of size $(2n+2)/4.$}\label{fig_path_eta}
\end{figure}
\vspace{0.6 cm}
\begin{figure}[h!]
  \centering
    \includegraphics[width=0.5\textwidth]{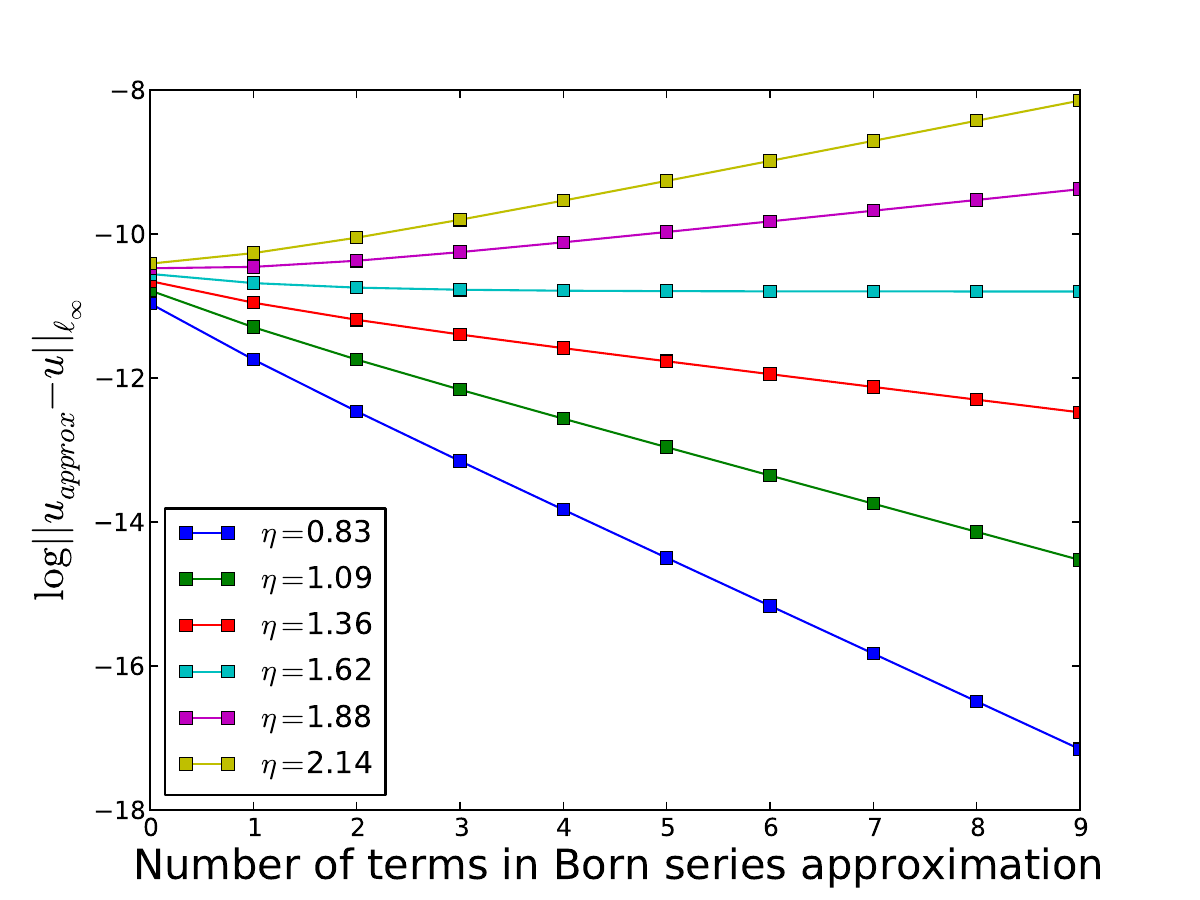}
      \caption{Plots of the $\ell_\infty$ error of the truncated solution $u_N.$ The green and red curves correspond to the bound on $\eta_{\rm max}$ from Theorem \ref{born_prop_1} and Proposition \ref{born_prop_2}, respectively. The other blue lines correspond to $\eta_{\rm max}$ spaced 0.026 apart.}\label{fig_path_error}
\end{figure}
\begin{table}[h!]
\centering
    \begin{tabular}{|l|l|}
    \hline
	{\bf Bound} & {\bf Cut-off} $\eta_{\rm max}$\\
	\hline
	Numerical Experiment & $1.15$\\
	Theorem \ref{born_prop_1} & 0.8333 \\
	Proposition \ref{born_prop_2}& 1.14655\\
    \hline
    \end{tabular}
    \caption{Comparison of theoretical bounds and experimental results for the maximum possible value of $\eta_{\rm max}$ for which the Neumann series converges.}\label{tab_path_cut_off}
 \end{table}

\subsection{Inhomogeneous absorption on a complete graph with boundary}
As a second example, using the background Green's function for the complete graph obtained in Appendix A, and listed in Table \ref{table:background_greens_functions}, we can once again solve the time-independent diffusion problem (\ref{eq:pet_diff}) for sufficiently small perturbations. In particular, choosing $\eta$ to be that given in Figure \ref{fig_comp_eta}, the associated errors for various values of $\eta_{\rm max}$ are given in Figure \ref{fig_comp_ell2}.
\begin{figure}[h!]
  \centering
    \includegraphics[width=0.5\textwidth]{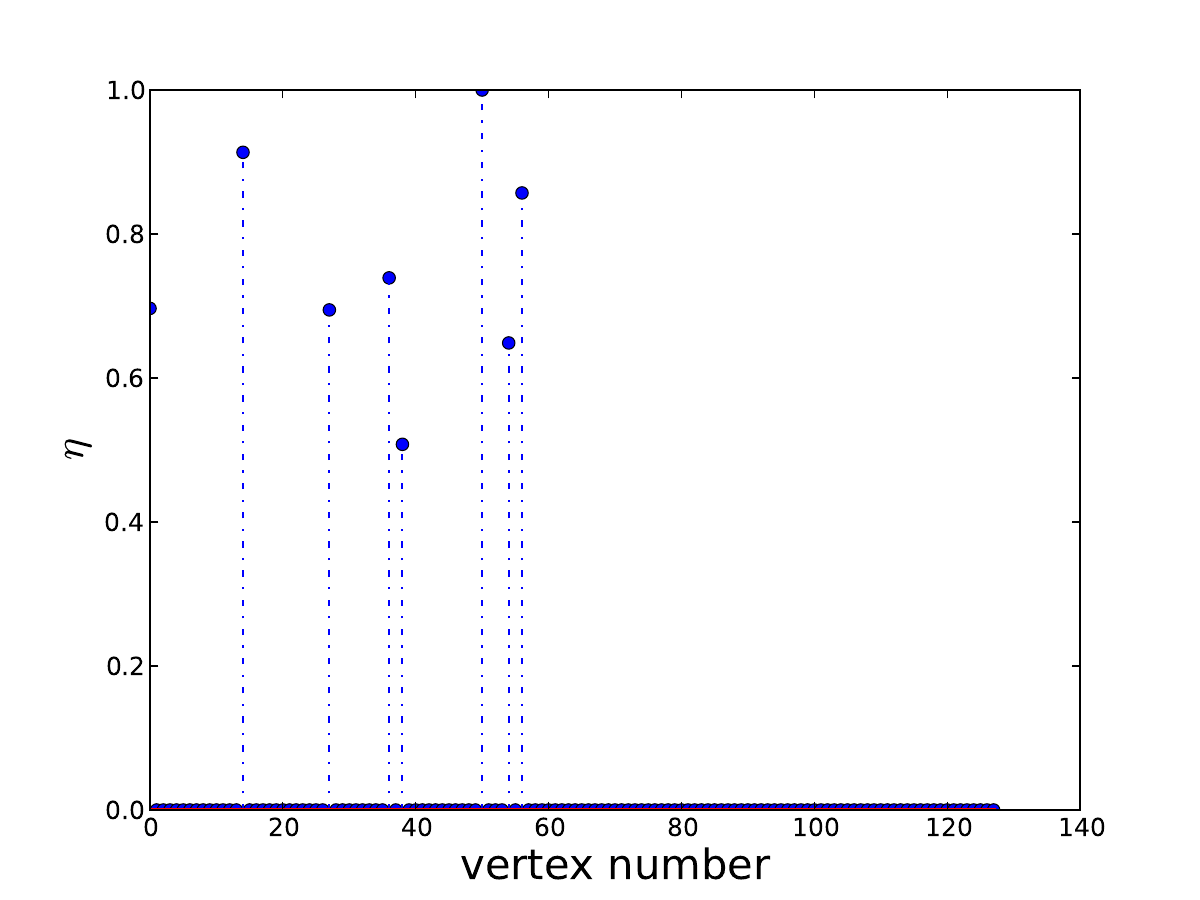}
      \caption{ A typical absorption vector $\eta$ used for the example of constructing a Green's function for the spatially-varying time-independent diffusion equation (\ref{eq:unpet_diff}) via Born series. The support of $\eta$ is once again chosen to be  a random sample of the interior vertices of size $d/4.$}\label{fig_comp_eta}
\end{figure}
\begin{figure}[h!]
  \centering
    \includegraphics[width=0.5\textwidth]{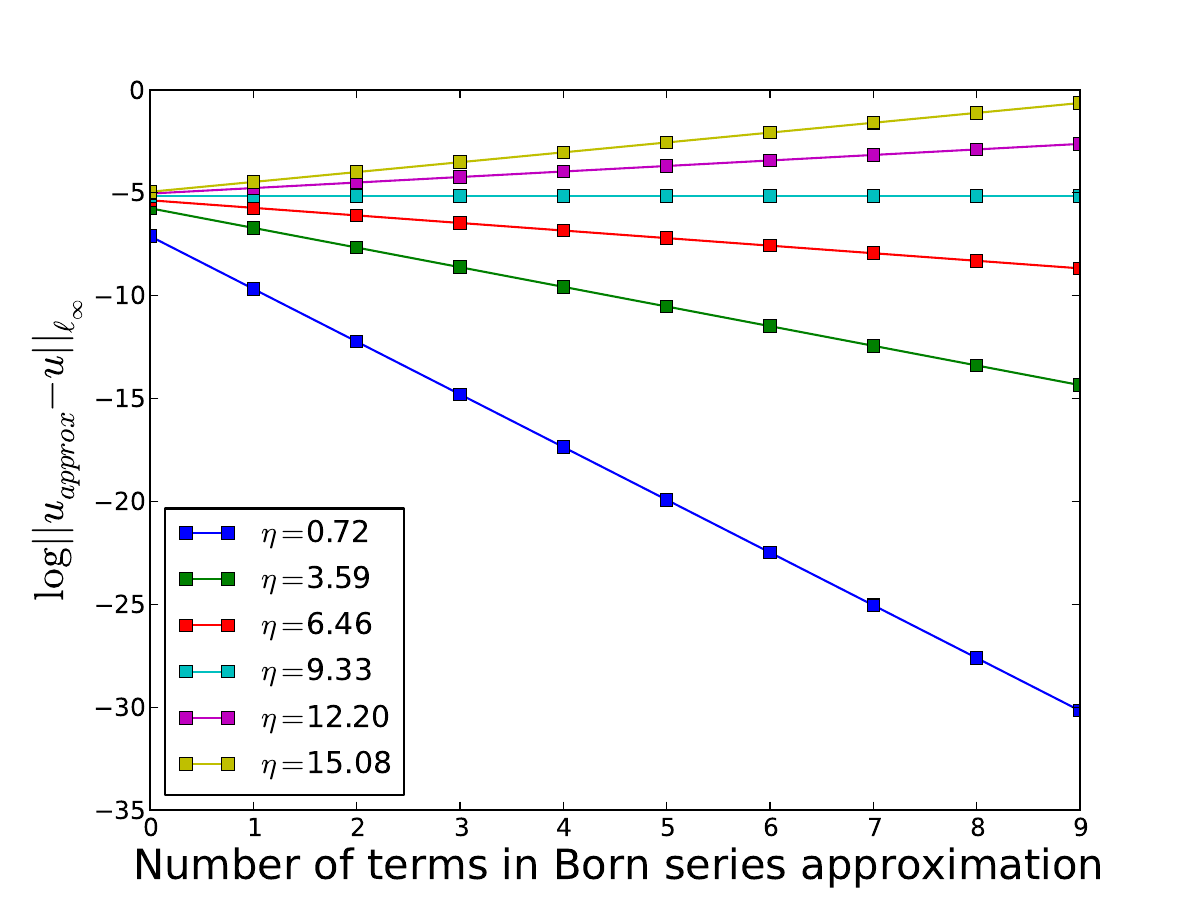}
      \caption{Plots of the $\ell_\infty$ error of the truncated solution $u_N.$ The green and red curves correspond to the bound on $\eta_{\rm max}$ from Theorem \ref{born_prop_1} and Proposition \ref{born_prop_2}, respectively.}\label{fig_comp_ell2}
\end{figure}

\clearpage
\section{Point Absorbers}\label{sec_pt_abs}
In scattering theory a classic problem is to consider a medium which is entirely homogeneous except for a few small inhomogeneities referred to as point absorbers \cite{foldy_lax}. For convenience we typically assume that the inhomogenieties are sufficiently far apart relative to their diameters, so that each can be thought of as being supported on a single point. 

\subsection{A single point absorber}
As above let $\Gamma=(V',E)$ be a graph, let $V \subset V'$ and let $\delta V$ be defined as in (\ref{eq:delta_def}). If a single point absorber is present then $\eta:(V\cup\delta V) \rightarrow \mathbb{R}$ is of the form
\begin{equation}\label{eq:eta_1ptabs}
\eta(x) = \kappa \, \delta_{y}, 
\end{equation}
where $y \in V$ is the location of the point absorber and $\kappa$ is a positive constant. For potentials of this form, we have the following theorem.

\begin{theorem}
Let $G_0$ be the background Green's function for the diffusion equation (\ref{eq:unpet_diff}). If the potential, $\eta,$ is due to a single point absorber located at the vertex labelled by $y,$ then the Green's function, $G,$ for (\ref{eq:pet_diff}) satisfies
\begin{equation}\label{eq:sing_pt_res}
G_0 -G =  \frac{\alpha_0 \kappa}{1+\alpha_0 \kappa \, G_0(y,y)} G_0(\cdot,y) G_0(\cdot,y)^*.
\end{equation}
\end{theorem}

\begin{proof}
Let $G$ denote the Green's function for the diffusion equation (\ref{eq:pet_diff}). By definition, we see that
\begin{equation}
G = H^{-1} = [H_0 + \alpha_0 \kappa e_y e_y^T]^{-1},
\end{equation}
and hence that $H$ is a rank one perturbation of the operator $H_0.$ The Sherman-Morrison formula, see \cite{disc_op} for example, then yields
\begin{equation}
(G_0 -G)(i,j) =\frac{\alpha_0 G_0^*(j,y)G_0(i,y)}{1+\alpha_0 \kappa G_0(y,y)},
\end{equation}
which completes the proof.
\end{proof}

For example, consider the infinite path whose background Green's function can be obtained by taking the limit as $n$ goes to infinity of the Green's function for the finite path given in Table \ref{table:background_greens_functions}. If the point absorber is located at the vertex $k$ then equation (\ref{eq:sing_pt_res}) yields
\begin{equation}\label{eq:pt_scat_1}
(G_0 -G)(i,j) = \frac{\alpha_0 \kappa}{1+\frac{\alpha_0 \kappa}{r-r^{-1}}} e^{-\log(r)\, \left(|i-k|+|j-k|\right)}.
\end{equation}
Note that unlike the continuous case \cite{foldy_lax}, no renormalization is required to obtain equation (\ref{eq:pt_scat_1}). This is an immediate consequence of the fact that in the discrete setting the operator $G_0$ is bounded for all $i$ and $j,$ whereas for the continuous problem $G_0$ is a singular integral operator. Physically, renormalization corresponds to giving each point absorber a non-zero size which we assume is small relative to the wavelength of the incident field. For the infinite path, since the system is discrete, the point absorbers automatically have a non-zero width and so no additional length scales need to be introduced.

As a second example, we consider the complete graph with boundary. For convenience we assume that the absorber is located in the interior of the graph and that the source and detector are located on the boundary of $V.$ Using the Green's function given in equation (\ref{eq:comp_green_fcn}) with the formula (\ref{eq:sing_pt_res}) we obtain
\begin{equation}
(G_0-G)(i,j) =\frac{\alpha_0 \kappa}{1+\frac{\alpha_0 \kappa \sigma}{(\sigma-1)(\sigma-1+d)} } \frac{\gamma^2 \sigma^2}{(\sigma-1)^2(\sigma-1+d)^2},  
\end{equation}
where $i,j \in \delta V,$ $\gamma = 1/(1+t)$ and $\sigma = 2+ \alpha_0 - \gamma.$ 

\subsection{Multiple point absorbers}
We now consider the case where there are $m$ identical point scatterers.
\begin{theorem}
Let $G_0$ be the background Green's function for the diffusion equation (\ref{eq:unpet_diff}). Further suppose that the potential, $\eta,$ consists of $m$ identical point absorbers of strength $\kappa,$ located at the vertices $\Lambda =\{x_{k_1},\ldots, x_{k_m}\} \subset V.$ Let $I_\Lambda$ be the $ (|V|+|\delta V|)\times m$ submatrix of the identity obtained by taking the columns of $I$ indexed by $\Lambda.$ The Green's function, $G,$ for (\ref{eq:pet_diff}) satisfies 
\begin{equation}\label{eq:mult_pt_res}
G_0 -G =  {\alpha_0 \kappa}G_0 I_\Lambda [I+\alpha_0 \kappa I_\Lambda^T G_{0}I_\Lambda]^{-1} I_\Lambda^TG_0.
\end{equation}
\end{theorem}

\begin{proof}
We begin by observing that by definition, $G$ satisfies
\begin{equation}
H_0 G =I  -\alpha_0 D_\eta  G.
\end{equation}
Using our definition of $I_\Lambda,$ we can rewrite this as
\begin{equation}
H_0 G = I - \alpha_0 \kappa I_\Lambda I_\Lambda^T G.
\end{equation}
Similarly, we observe that
\begin{equation}\label{eq:useful_eq}
G = G_0 - \alpha_0 \kappa G I_\Lambda I_\Lambda^T G_0,
\end{equation}
from which it follows that
\begin{equation}
H_0 G = I - \alpha_0\kappa I_\Lambda I_\Lambda^T [G_0 - \alpha _0 \kappa G I_\Lambda I_\Lambda^T G_0 ],
\end{equation}
where we have used the fact that
\begin{equation}
H_0 G_0= I.
\end{equation}
Thus
\begin{equation}
(G_0-G)=\alpha_0 \kappa \left[ G_0 (I_\Lambda I_\Lambda^T) G_0 - \alpha_0 \kappa G_0 I_\Lambda (I_\Lambda^T G I_\Lambda) I_\Lambda^T G_0\right],
\end{equation}
and hence
\begin{equation}
G_0 -G = \alpha_0 \kappa (G_0 I_\Lambda) \left[ I_\Lambda^T I_\Lambda - \alpha_0 \kappa (I_\Lambda^T G I_\Lambda) \right] (I_\Lambda^T G_0).
\end{equation}
To find $I_\Lambda^T G I_\Lambda$ we left muliply (\ref{eq:useful_eq}) by $I_\Lambda^T$ and right multiply by $I_\Lambda$ to obtain
\begin{equation}
G' = G_0' -\alpha_0 \kappa G' G_0',
\end{equation}
where $G_0' = I_\Lambda^T G_0 I_\Lambda$ and $G' = I_\Lambda^T G I_\Lambda.$ It follows that
\begin{equation}\label{eq:mult_fl_scat_m}
G' =G_0' [I_\Lambda^T I_\Lambda + \alpha_0 \kappa G_0']^{-1}. 
\end{equation}
Note that the inverse of the matrix in  (\ref{eq:mult_fl_scat_m}) exists since the smallest eigenvalue  of $G_0'$ must be greater than the smallest eigenvalue of $G_0,$ which is positive definite. Since $I_\Lambda^T I_\Lambda$ is the $m \times m$ identity matrix it follows that $I_\Lambda^T I_\Lambda + \alpha_0 \kappa G_0'$ has positive eigenvalues and is therefore invertible.

 Finally, observe that 
\begin{equation}
I_\Lambda^T I_\Lambda  - \alpha_0 \kappa G_0' [I_\Lambda^T I_\Lambda + \alpha_0 \kappa G_0']^{-1} =[I_\Lambda^T I_\Lambda + \alpha_0 \kappa G_0']^{-1},
\end{equation}
from which it follows that
\begin{equation}
G_0 -G =\alpha_0 \kappa (G_0 I_\Lambda) \left[ I_\Lambda^T I_\Lambda + \alpha_0 \kappa I_\Lambda^TG_0 I_\Lambda \right]^{-1} (I_\Lambda^T G_0),
\end{equation}
which completes the proof. Alternatively we could iterate the Sherman-Morrison formula, or follow the approach of \cite{disc_op}.
\end{proof}

As an example, we once again consider the infinite path and suppose there are two point absorbers located at the vertices corresponding to $k_1$ and $k_2.$ Here
\begin{equation}
G_0'  = \frac{1}{2 \sinh \log r} \left(\begin{array}{cc}
1& e^{-\log(r) |k_2-k_1|}\\
e^{-\log(r)|k_2 -k_1|} & 1
\end{array}\right).
\end{equation}
If $f_{i,j} = e^{-\log(r) |i-j|}/{2 \sinh\log{r}}$ and $s = \sinh(\log(r))$ a straightforward calculation yields
\begin{equation}\label{eq:2_pt_abs_final}
\begin{split}
e_j^*(G_0-G)e_i = &\frac{1}{2}\left( \begin{array}{cc} f_{j,k_1}+f_{j,k_2} &f_{j,k_1}-f_{j,k_2} \end{array} \right)
\left(\begin{array}{cc}\frac{\alpha_0\kappa}{1+\alpha_0\kappa(s^{-1}+f_{k_1,k_2})}&0\\0&\frac{\alpha_0\kappa}{1+\alpha_0\kappa(s^{-1}-f_{k_1,k_2})}  \end{array} \right)\\
& \times \left( \begin{array}{cc} f_{i,k_1}+f_{i,k_2} \\ f_{i,k_1}-f_{i,k_2} \end{array} \right).
 \end{split}
\end{equation}
Observe that in the limit as $|k_1 -k_2| \rightarrow \infty,$ $f_{k_1,k_2} = o(1)$ and hence equation (\ref{eq:2_pt_abs_final}) becomes
\begin{equation}
\begin{split}
e_j^*(G_0-G)e_i \approx &\frac{\alpha_0 \kappa}{2(1+\alpha_0 \kappa s^{-1})}\left( \begin{array}{cc} f_{j,k_1}+f_{j,k_2} &f_{j,k_1}-f_{j,k_2} \end{array} \right) \left( \begin{array}{cc} f_{i,k_1}+f_{i,k_2} \\ f_{i,k_1}-f_{i,k_2} \end{array} \right)\\
=&\frac{\alpha_0 \kappa}{2(1+\alpha_0 \kappa s^{-1})}\left[f_{i, k_1} f_{j,k_1}+f_{i,k_2}f_{j,k_2}\right]\\
=& G_1(i,j; k_1)+G_1(i,j;k_2)
\end{split}
\end{equation}
where $G_1(i,j;_k)$ is the Green's function for one point absorber located at the point $k.$ Thus as the separation of the two point absorbers increases, the Green's function tends toward the sum of the Green's functions for two non-interacting point absorbers. Sample plots are shown in Figure \ref{fig_examp_pt_abs} for the infinite path with two point absorbers equidistant from a point source located at the origin. Here $u_0$ represents the solution to the homogeneous problem and $u$ denotes the solution to the full time-independent diffusion equation.
\vspace{0.6 cm}

\begin{figure}[h!]
  \centering
    \includegraphics[width=0.5\textwidth]{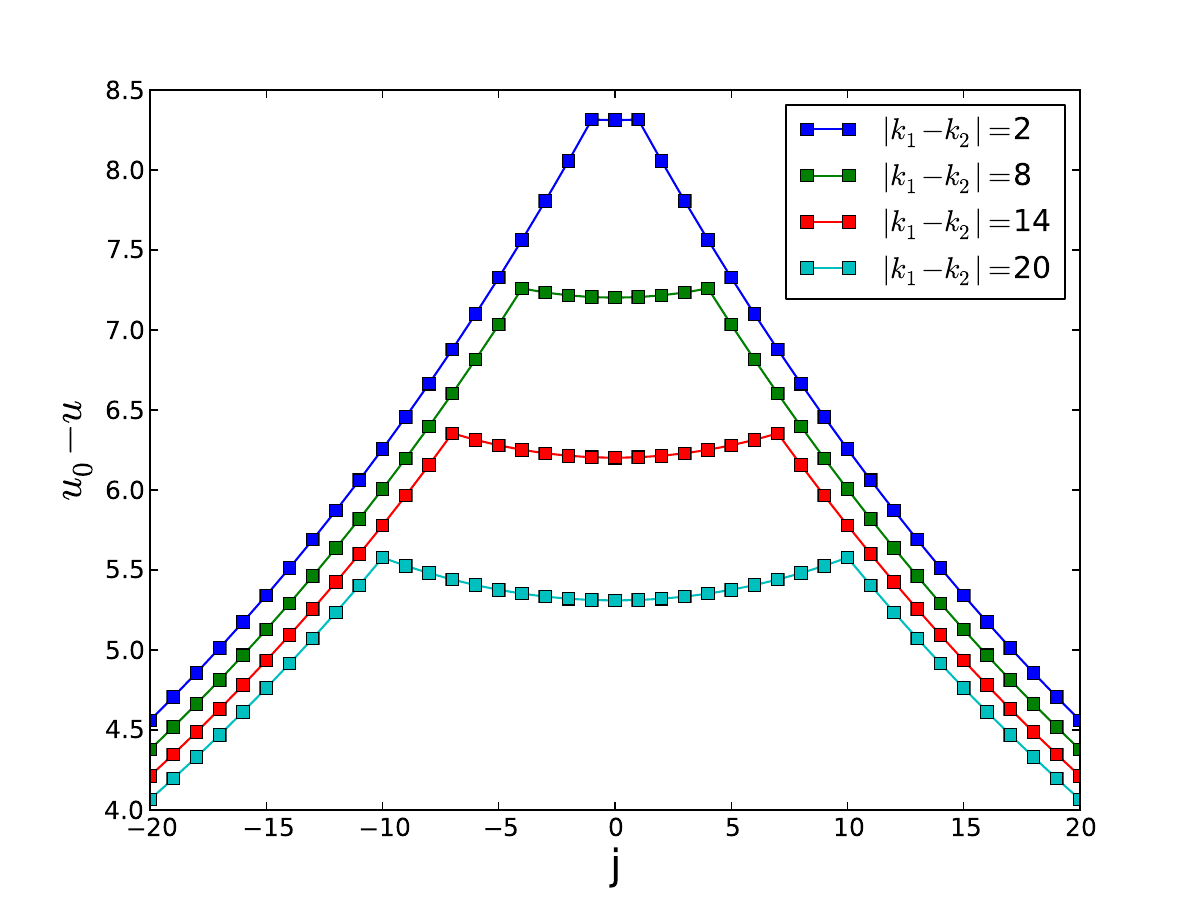}
      \caption{Plots of $u_0-u$ for the infinite path with two identical point scatterers equidistant from a point source located at the origin with $\alpha_0 = 0.001$ and $\kappa = 100.$}\label{fig_examp_pt_abs}
\end{figure}
\clearpage

As a second example we consider the scattering from two point absorbers on the infinite two-dimensional lattice. An analysis of the scattering properties of this system requires an expression for the Green's function of the diffusion equation (\ref{eq:unpet_diff}) on $\mathbb{Z}\times \mathbb{Z},$ an integral formula for which was found in Proposition \ref{2d_latt} of Appendix A. For simplicity we specialize to the case in which the are two point absorbers are positioned on the y-axis and the source and detector are located on the x-axis as in Figure \ref{fig_lat_di}. 
\begin{figure}[h!]
\centering
\includegraphics[width = 0.5\textwidth]{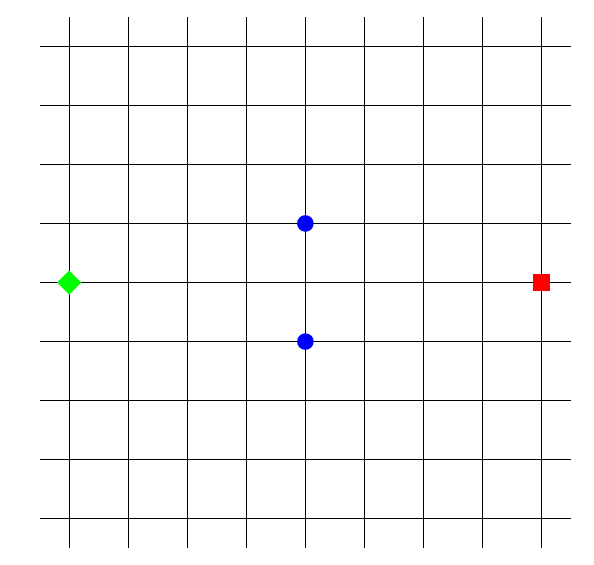}
   \caption{A diagrammatic representation of the geometry of the point absorbers (blue circles), source (green diamond) and detector (red square) used to study the scattering properties of the two point absorber system on the infinite square lattice.}\label{fig_lat_di}
   \end{figure}
   
If we assume the source is located at $(s,0)$ the detector is at $(j,0),$ and the point absorbers are located at $(0, k_1),$ and $(0,k_2),$ then an analogous calculation to the one performed for the one-dimensional case yields
\begin{equation}\label{eq:2_pt_abs_2d_final}
\begin{split}
e_j^*(G_0-G)e_s = &\frac{1}{2}\left( \begin{array}{cc} f_{j,k_1}+f_{j,k_2} &f_{j,k_1}-f_{j,k_2} \end{array} \right)
\left(\begin{array}{cc}\frac{\alpha_0\kappa}{1+\alpha_0\kappa \lambda_+}&0\\0&\frac{\alpha_0\kappa}{1+\alpha_0\kappa \lambda_-}  \end{array} \right)\\
& \times \left( \begin{array}{cc} f_{s,k_1}+f_{s,k_2} \\ f_{s,k_1}-f_{s,k_2} \end{array} \right).
 \end{split}
\end{equation}
where 
\begin{equation}\label{eq:lam_pm_2pt}
\begin{split}
\lambda_{\pm} &= g(0,0) \pm g(0,k_2-k_1) \\
&= \frac{1}{\pi a} K\left( \frac{1}{a^2}\right) \pm  \frac{1}{2\pi} \int_{0}^\pi \frac{\cos[(k_2-k_1)\,v]\, (\cos{v})^{|k_2-k_1|}}{(a+\sqrt{a^2-\cos^2 v})^{|k_2-k_1|} \sqrt{a^2-\cos^2{v}}}\, {\rm d}v
\end{split}
\end{equation}
and $f_{s,k} = g(|s|,|k|).$ Results for $-s = j = 1$ and $k_1= -k_2$ are shown in Figure \ref{fig:lat_res} for various values of  the point absorber separation $|k_2-k_1|$ with $\alpha_0 = 10^{-3}$ and $\kappa =  10^3.$ Note that due to the nature of our expression for the isotropic Green's function we cannot evaluate equation (\ref{eq:2_pt_abs_2d_final}) exactly and must make use of numerical integration both to evaluate $g(|s|,|k|)$ and to find values for the integrals in (\ref{eq:lam_pm_2pt}).
\vspace{0.5 cm}

\begin{figure}[h!]
  \centering
    \includegraphics[width=0.5\textwidth]{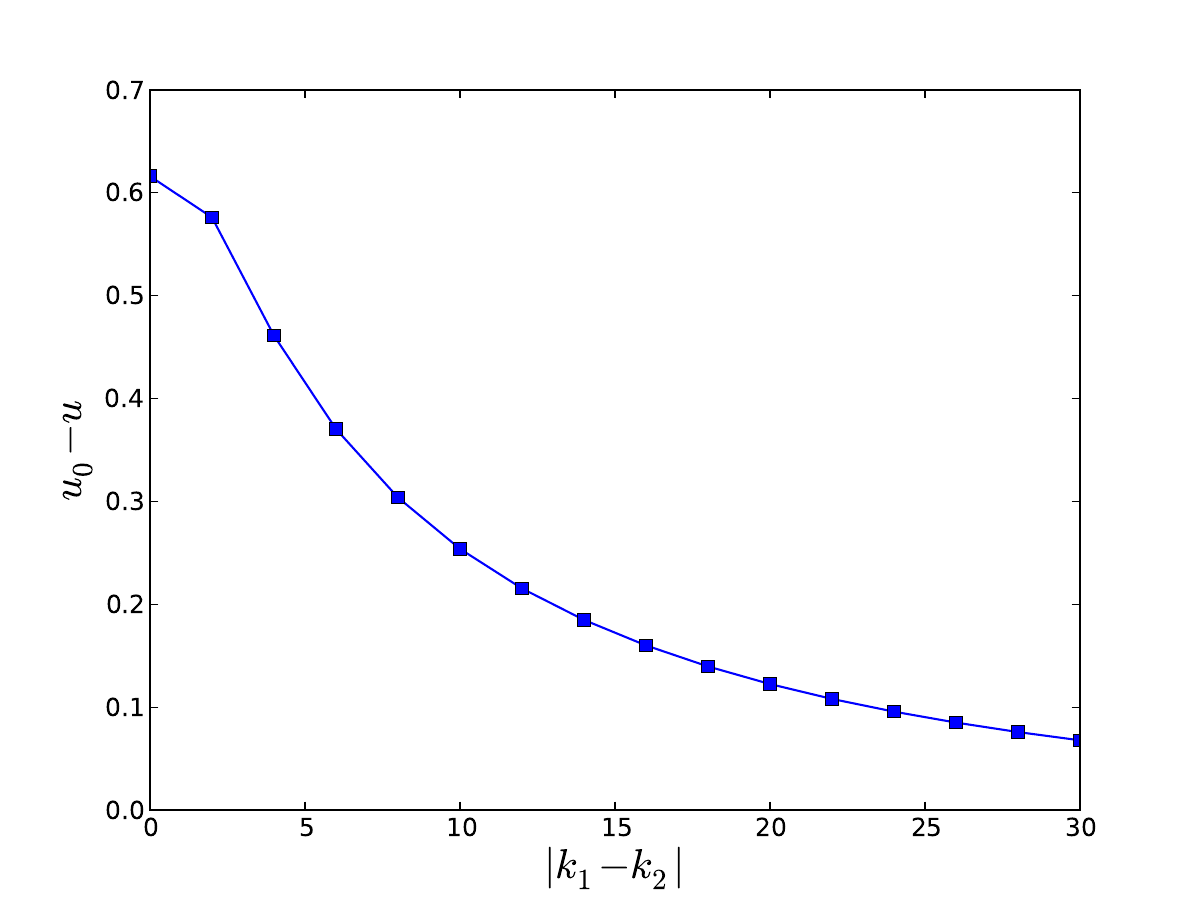}
      \caption{Plots of $u_0-u$ for the infinite plane with two identical point scatterers equidistant from a point source located at $(-1,0)$ and detector located at $(1,0)$ with $\alpha_0 = 10^{-3}$ and $\kappa = 10^{3}.$}\label{fig:lat_res}
\end{figure}

\section{Acknowledgements}
This work was supported in part by the NSF grants DMSÐ1115574 and DMSÐ1108969 to JCS and NSF CCF 1161233 and NSF CIF 0910765 to ACG.

\bibliography{bib_file}{}
\bibliographystyle{hsiam}

\begin{appendices}
\section{Computation of background Green's functions}
\subsection{Analysis of a M\"{o}bius ladder}
Another family of vertex-transitive graphs of particular interest in material science \cite{proc_knot_92,Mob_Lad_Paper,nat_mob_lad} and computer science \cite{lin_ord_poly_mob} are the M\"{o}bius ladders on $2n+2$ vertices. Using a similar approach as above we can compute the background Green's function  for the diffusion equation on this family of graphs. For convenience we assume $n$ is odd though a similar result can be obtained in the even case by a slight modification to the proof of the following theorem.

\begin{theorem}
\label{prop_mob}
Consider the M\"{o}bius ladder with $2n+2$ vertices $\{0, 1 ,\ldots, 2n+1\},$ $n$ odd, as shown in Figure \ref{fig_mob}. The associated Green's function for the diffusion equation with uniform absorption (\ref{eq:unpet_diff}) is
\begin{equation}
G(i,j) = \left\{ \begin{array}{lr} g_1(|i-j|_{\rm min}) + g_2(|i-j|_{\rm min}), \quad & |i-j|\le \frac{n+1}{2}+1  \\
 g_1(|i-j|_{\rm min})-g_2(|i-j|_{\rm min}), \quad & |i-j| > \frac{n+1}{2}\end{array} \right.
\end{equation}
for all $0 \le i,j \le 2n+1,$ where $|i-j|_{\rm min} = \min\{ |i-j|, \, 2n+2-|i-j|,\, \left|\,|i-j| - (n+1) \right| \},$
\begin{equation}
g_k(s) = \frac{(a_k r_k^{(n+1)/2}-a_k^{-1} r_k^{-(n+1)/2})(a_k r_k^{(n+1)/2-s}-a_k^{-1} r_k^{-[(n+1)/2-s]}) }{(r_k-r_k^{-1}) (a_k^2 r_k^{n+1}-a_k^{-2} r_k^{-(n+1)})}, \quad k=1,2,
\end{equation} 
$r_k$ satisfies $r_k+1/r_k = 2k+\alpha_0,$
\begin{equation}
a_1 = \left[1+\frac{r^2_1-1}{r_1(1+\frac{\alpha_0}{2}-r_1)} \right]^\frac{1}{2},
\end{equation}
and $a_2 = 1.$
\end{theorem}

\begin{figure}[!!h]
\footnotesize
\centering
\includegraphics{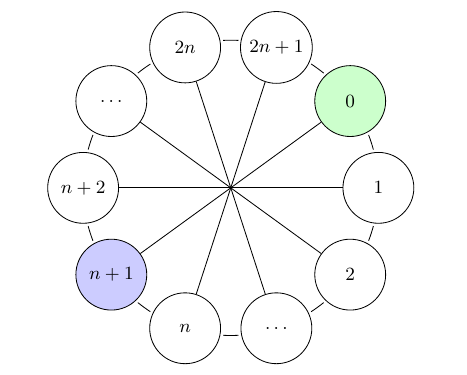}
\caption{The M\"{o}bius ladder with $2n+2$ vertices. }\label{fig_mob}
\end{figure}

\begin{proof}
The result is proved in a manner similar to the method of images used in PDEs. We decompose the Green's function into two functions one of which is symmetric and the other antisymmetric with respect to reflection through the origin. Unlike in the case of method of images, however, the `mirror charges' are located within the domain of interest and we rely on cancellations to recover the desired solution. In considering the symmetries of the problem, it is clear that $G(i,j)$ must depend only on the number of vertices of the loop lying between the two vertices $x_i$ and $x_j,$ since the graph $\Sigma$ is invariant under cyclic permutations of its vertices. Thus, without loss of generality, we can fix $j = 0.$ Letting $H_0$ be the operator associated with the homogeneous time-independent diffusion equation (\ref{eq:unpet_diff}) we now consider two related problems:
\begin{enumerate}[(A)]
\item find the vector $g_2 \in \mathbb{R}^{2n+2}$ such that $H_0 \,g_2 = \frac{1}{2}(e_0 - e_{n+1}),$ where $e_k$ is the $k$th canonical basis vector
\item find the vector $g_1 \in \mathbb{R}^{2n+2}$ such that $H_0 \,g_1 = \frac{1}{2}(e_0 + e_{n+1}).$ 
\end{enumerate}

In considering subproblem (A), we see by inspection that if $g_2$ is a solution then it must satisfy
\begin{equation}
g_2(i) = g_2(2n+2-i) = - g_2(n+1-i) = -g_2(n+1+i),
\end{equation}
where once again for ease of notation we take all indices modulo $2n+2.$ It follows immediately that
\begin{equation}
g_2((n+1)/2) = g_2(-(n+1)/2) = 0.
\end{equation}
Moreover, applying $H_0$ we see that
\begin{equation}
\begin{split}
\frac{1}{2}\delta_{0,i} &= (3+ \alpha_0) g_2(i) -g_2(i+n+1)-g_2(i+1)-g_2(i-1),\\
&= (4+\alpha_0) g_2(i) - g_2(i+1) - g_2(i-1),
\end{split}
\end{equation}
for $-(n+1)/2 < i < (n+1)/2.$ Hence $2 g_2(i),$ $-(n+1)/2< i < (n+1)/2$ satisfies the same equation as the Green's function for the centered path with source at $j=0,$ $\alpha_0$ replaced by $2+ \alpha_0$ and with Dirichlet boundary conditions at $i = \pm(n+1)/2.$ A slight modification to the first example of Table \ref{table:background_greens_functions} then yields
\begin{equation}
g_2(i) = \frac{(r^{(n+1)/2}-r^{-(n+1)/2})(r^{(n+1)/2-i}-r^{-(n+1)/2+i})}{2(r-r^{-1})(r^{n+1}-r^{-(n+1)})}
\end{equation}
where $r + r^{-1} = 4 + \alpha_0.$

To solve subproblem (B), we begin by noting that
\begin{equation}
g_1(i) = g_1(2n+2-i)=g_1(n+1-i) = g_1(n+1+i),
\end{equation}
for all $i = 0,\ldots, 2n+1$ and where all indices are taken modulo 2n+2. It follows that it is sufficient to find $g_1(i)$ for $ i =-(n+1)/2,\ldots,-2,-1,0,1,2,\ldots,(n+1)/2.$ By symmetry we know that $g_1((n+1)/2-1) = g_1((n+1)/2+1)$ and $g_1(-(n+1)/2) = g_1((n+1)/2)$ so that
\begin{equation}
\begin{split}
0 &=(3+ \alpha_0) g_1((n+1)/2) - g_1(-(n+1)/2)-g_1((n+1)/2-1)-g_1((n+1)/2+1) \\
&=(2+ \alpha_0) g_1((n+1)/2)-2 g_1((n+1)/2-1).
\end{split}
\end{equation}
Similar reasoning applies to $g_1(-(n+1)/2)$ and hence
\begin{equation}
(1+\frac{\alpha_0}{2}) g_1(\frac{n+1}{2}) - g_1(\frac{n+1}{2}-1) = 0, \,\, {\rm and} \,\, (1+\frac{\alpha_0}{2}) g_1(-\frac{n+1}{2}) - g_1(-\frac{n+1}{2}+1) = 0.
\end{equation}
For all $i,$ $-(n+1)/2<i<(n+1)/2$ we find from the above symmetries that
\begin{equation}
(2+\alpha_0) g_1(i) -g_1(i-1)-g_1(i+1) = \frac{1}{2} \delta_{i,0}.
\end{equation}
It follows immediately that the equations satisfied by $2 g_1(i),$ $i = -(n+1)/2,\ldots,(n+1)/2$ are identical to those defining the Green's function for the centered path with source at $j=0$ and with Robin boundary conditions $t = \alpha_0/2.$ Once again using a slight modification of the first line of Table \ref{table:background_greens_functions}, we find that if $r + 1/r = 2 + \alpha_0,$ $t = \alpha_0/2$ and $a = \left[1 + \frac{r^2-1}{r(1+t-r)}\right]^\frac{1}{2}$ then
\begin{equation}
g_1(i) = \frac{(a r^{(n+1)/2}-a^{-1} r^{-(n+1)/2})(a r^{(n+1)/2-i}-a^{-1} r^{-(n+1)/2+i})}{2(r-r^{-1}) (a^2 r^{n+1} - a^{-2} r^{-(n+1)})}.
\end{equation}
The remaining entries can then be found immediately by reflection.

\end{proof}

\clearpage

\subsection{Analysis of the complete graph on $d$ vertices}
We can perform a similar computation with the complete graph on $d$ vertices, an example of which is shown in Figure \ref{fig_complete} for $d=10.$ 

\begin{proposition}
\label{prop_complete_gr}
Let $R$ be the complete graph on $d$ vertices $\{0, 1 ,\ldots, d-1\}$  with $d$ boundary vertices $\{0',\ldots, (d-1)'\}.$ The associated Green's function for the diffusion equation (\ref{eq:unpet_diff}) with Robin boundary conditions is
\begin{equation}\label{eq:comp_green_fcn}
   G(x,y) = \left\{
     \begin{array}{lr}
       \frac{\sigma}{(\sigma-1)(\sigma-1+d)} &{\rm if}\,\, x =y \in V,\\
       
       \frac{1}{(\sigma-1)(\sigma-1+d)} &{\rm if}\,\, x \neq y,\, x,y \in V,\\
       \frac{\gamma^2 \sigma}{(\sigma-1)(\sigma-1+d)}+\gamma &{\rm if}\,\, x = y\in \delta V,\\
     \end{array}
   \right.
\end{equation} 
where $\gamma = 1/(1+t)$ and $\sigma = 2+ \alpha_0 - \gamma.$ The remaining entries can be obtained via the Robin boundary conditions and the identity 
\begin{equation}
G(x,y) = G(y,x)
\end{equation}
which holds for all $x,y \in V \cup \delta V.$
\end{proposition}

\begin{figure}[!!h]
\centering
\includegraphics{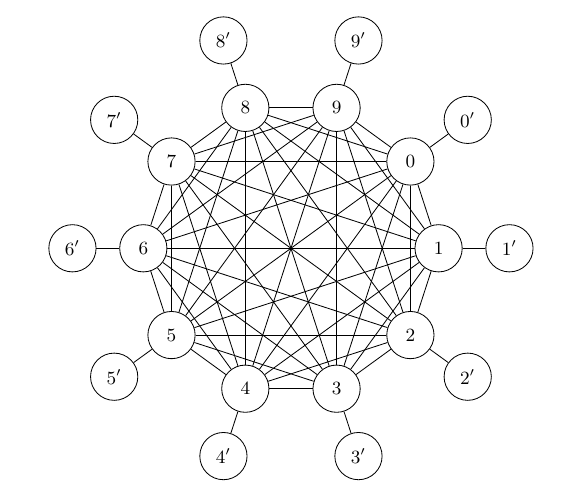}
\caption{The complete graph on $10$ vertices with $10$ boundary vertices. }\label{fig_complete}
\end{figure}

\begin{proof}
We first consider the case where the source is located at an interior vertex. By symmetry we may assume without loss of generality that $y = 0.$ Note that upon fixing $y=0$ the graph is invariant under a permutation of all remaining vertices, provided it preserves the edge between the vertex in $R$ and the corresponding boundary point.




If $x \neq 0 $ we observe that

\begin{equation}
\left[ (d +\alpha_0) G({x},0) - (d-2) G({x},0) - G(0,0) - G({x}',0)\right] = 0 .
\end{equation}
Using Robin boundary conditions we see that $G(x',0) = G({x},0)/(1+t)$ and hence
\begin{equation}
G(0,0) = \left[ 2+\alpha_0-\gamma \right]\, G(x,0).
\end{equation}
Let $g = G(x,0)$ and $\sigma = \left[ 2+\alpha_0-\gamma \right]$ in which case we obtain
\begin{equation}
\begin{split}
1&=  \left[ (d+ \alpha_0) G({0},0) - (d-1) g - \gamma G(0,0)\right] \\
&= \left[( d+\alpha_0) \sigma \, g - (d-1) g - \gamma \sigma g \right]\\
&= {g} (\sigma-1) (\sigma+d-1)
 \end{split}
 \end{equation}
 and thus
\begin{equation}
g = \frac{1}{(\sigma-1)(\sigma+d-1)}.
\end{equation}
 The remainder of the result follows immediately from noting that $G(0,0) = \sigma g$ and by using the symmetries of the complete graph described above.

 Now suppose the source is located on the boundary. Again, without loss of generality, we may assume that the source is located at the vertex $0'.$ If $x \neq 0$ then
 \begin{equation}
 0= [(d+\alpha_0)\,G(x,0') - (d-2)G(x,0')-G(0,0')-G(x',0)].
 \end{equation}
If $G(x,0') = g$ then $G(0,0') = (2+\alpha_0-\gamma)g=\sigma g.$ If $x= 0$ then
\begin{equation}
0 = [(d+\alpha_0) \sigma g -(d-1) g - G(0',0')].
\end{equation}
The Robin boundary condition $t G(0',0') +[G(0',0')-G(0,0')] =1$ implies that
\begin{equation}
g = \gamma \frac{1}{(d+\alpha_0)\sigma - (d-1) -\sigma \gamma}= \frac{\gamma}{(\sigma-1)(\sigma+d-1)}.
\end{equation} 
 \end{proof}
 \clearpage

\subsection{Analysis of a two-dimensional lattice}

We conclude our catalogue of examples with a discussion of the Green's function for the two-dimensional lattice $\Sigma = \mathbb{Z} \times \mathbb{Z}.$ For convenience, we index the vertices with ordered tuples $V = \{(m,n)\, |\, m,n \in \mathbb{Z}\}$ and hence if $x = (m_1,n_1)$ and $y= (m_2,n_2)$ are two vertices then $x \sim y$ if and only if $|m_2-m_1|+|n_2-n_1| =1.$ In the following proposition we obtain an integral representation of the Green's function for the isotropic time-independent diffusion equation on the infinite two-dimensional lattice by means of a discrete Fourier transform.

\begin{proposition}
\label{2d_latt}
Let $\Gamma$ be the graph $\mathbb{Z} \times \mathbb{Z}$ with vertices labelled by $\{(m,n)\, |\, m,n \in \mathbb{Z}\}.$ The Green's function for the corresponding homogeneous time-independent diffusion equation (\ref{eq:unpet_diff}) is
\begin{equation}\label{eq:lat_2}
G((m_1,n_1),(m_2,n_2))= \frac{1}{2\pi} \int_0^\pi \frac{\cos\left( d_- \,v\right)\, (\cos(v))^{d_+}}{(a+\sqrt{a^2-\cos(v)})^{d_+} \sqrt{a^2 - \cos(v)}}\,{\rm d}v
\end{equation}
where $a = 1+\alpha_0/4,$ and $d_\pm = |m_2-m_1| \pm |n_2-n_1|.$ In particular, if $(m_1,n_1) = (m_2,n_2)$ then
\begin{equation}\label{eq:lat_2_cor}
G((m,n),(m,n)) = \frac{1}{\pi a} \,K\left(\frac{1}{a^2} \right)
\end{equation}
where $K$ is the complete elliptic integral of the first kind defined by \cite{ab_steg}
\begin{equation}
K(m) = \int_0^\frac{\pi}{2} \frac{1}{\sqrt{1- m^2 \sin^2 \phi}} {\rm d}\phi,
\end{equation}
for $m^2 <1.$
\end{proposition}
\begin{proof}
The approach for finding the Green's function is similar to that used for the Helmholtz equation \cite{econo, martin} and the Poisson equation \cite{soardi} on lattices. We begin by noting that the problem is invariant under translations and reflections, from which it follows that $G$ must only depend on the quantities $m = |m_2-m_1|$ and $n = |n_2-n_1|.$ Hence
\begin{equation}
G((m_1,n_1),(m_2,n_2)) = G((m,n),(0,0))  = g(m,n)
\end{equation}
for some function $g(m,n) \in \ell^2(\mathbb{Z}^2).$ Applying the operator $H_0$ defined in (\ref{eq:h_0_def}), we see that $g(m,n)$ satisfies the difference equation
\begin{equation}\label{eq:lat_diff}
(4 +\alpha_0) \,g(m,n) - g(m-1,n)-g(m+1,n)-g(m,n-1)-g(m,n+1) = \delta_{m,0} \delta_{n,0}.
\end{equation}
We next consider the discrete Fourier transform of (\ref{eq:lat_diff}). On $\mathbb{Z} \times \mathbb{Z}$ the Fourier transform $\mathcal{F}: \ell^1 (\mathbb{Z} \times \mathbb{Z})\rightarrow L^1((-\pi,\pi]^2)$ of a function $f$ is given by
\begin{equation}
\hat{f}(\xi,\eta) = \mathcal{F}\, (f)(\xi,\eta) = \sum_{n,m \in\mathbb Z} e^{-i \xi m-i \eta n} f(m,n).
\end{equation}
Thus, upon taking the Fourier transform of equation (\ref{eq:lat_diff}), we obtain
\begin{equation}
\left[(4+\alpha_0) -e^{i \xi}-e^{-i \xi} - e^{i \eta} - e^{-i \eta} \right]\, \hat{g} (\xi,\eta) = 1,
\end{equation}
where $\xi, \eta \in (-\pi,\pi].$ Using the identity
\begin{equation}
e^{i \xi}+e^{-i \xi} + e^{i \eta}+ e^{-i \eta} = (e^{i (\xi+\eta)/2}+e^{-i( \xi+\eta)/2}) (e^{i (\xi-\eta)/2}+e^{-i (\xi-\eta)/2})
\end{equation}
yields
\begin{equation}
\hat{g}(\xi,\eta) = \frac{1}{4 \left[ 1+\alpha_0/4-\cos\left(\frac{\xi+\eta}{2}\right) \cos \left(\frac{\xi-\eta}{2} \right)\right]}.
\end{equation}
Upon application of the inverse Fourier transform we find
\begin{equation}
g(m,n) = \frac{1}{(2\pi)^2} \int_{-\pi}^\pi \int_{-\pi}^\pi \frac{e^{i m \xi+i n \eta} }{4\left[ 1+\alpha_0/4-\cos\left(\frac{\xi+\eta}{2}\right) \cos \left(\frac{\xi-\eta}{2} \right)\right]}\,{\rm d}\xi\,\, {\rm d}\eta.
\end{equation}
If we change variables, letting $u = (\xi+\eta)/2$ and $v = (\xi-\eta)/2,$ we obtain
\begin{equation}
g(m,n)=\frac{1}{2(2\pi)^2} \int_{-\pi}^\pi e^{i(m-n)v} \int_{-\pi}^\pi e^{i(m+n) u} \frac{1}{1+\frac{\alpha_0}{4}- \cos{u}\, \cos{v}} \, {\rm d}u \, {\rm d}v.
\end{equation}
If we let $a = 1+\alpha_0/4$ and choose $z = e^{i u}$ we obtain
\begin{equation}
g(m,n) = \frac{1}{2(2\pi)^2 i} \int_{-\pi}^\pi e^{i (m-n) v} \oint_{C_1} \frac{z^{m+n}}{az - \frac{\cos{v}}{2}(z^2+1)} \,{\rm d}z\, {\rm d}v
\end{equation}
where $C_1$ is the unit circle oriented counterclockwise. Integration then yields
\begin{equation}
g(m,n) = \frac{1}{4 \pi} \int_{-\pi}^\pi e^{i(m-n)v}\frac{(\cos{v})^{m+n}}{(a+\sqrt{a^2-\cos^2{v}})^{m+n}\sqrt{a^2-\cos^2{v}}} \,{\rm d}v.
\end{equation}
Using the fact that the above expression is the Fourier transform of an even function we can re-write it as a real integral, yielding
\begin{equation}
g(m,n) = \frac{1}{2\pi} \int_{0}^\pi \frac{\cos[(m-n)v]\, (\cos{v})^{m+n}}{(a+\sqrt{a^2-\cos^2 v})^{m+n} \sqrt{a^2-\cos^2{v}}}\, {\rm d}v.
\end{equation}
The expression for the Green's function given in (\ref{eq:lat_2}) follows by employing the translation and reflection symmetries outlined above. To obtain the expression (\ref{eq:lat_2_cor}) we observe that $(m_2,n_2) = (m_1,n_1)$ corresponds to $m=n=0$ and hence is given by
\begin{equation}
\begin{split}
g(0,0) &= \frac{1}{2\pi} \int_{0}^\pi \frac{1}{ \sqrt{a^2-\cos^2{v}}}\, {\rm d}v\\
& = \frac{1}{2\pi a} \int_{-\frac{\pi}{2}}^\frac{\pi}{2} \frac{1}{\sqrt{1-(1/a)^2 \cos^2 \phi}}\,{\rm d} \phi\\
&= \frac{1}{\pi a} \int_0^\frac{\pi}{2} \frac{1}{\sqrt{1-(1/a)^2 \cos^2\phi}}\,{\rm d}\phi\\
&= \frac{1}{\pi a} K\left(\frac{1}{a^2}\right).
\end{split}
\end{equation}
\end{proof}
\end{appendices}
\end{document}